\documentclass[reqno,10pt]{amsart}

\usepackage[a4paper,left=35mm,right=35mm,top=30mm,bottom=30mm,marginpar=25mm]{geometry}

\usepackage{tikz}
\usepackage{enumitem}
\usepackage{amsmath}
\usepackage{amssymb}
\usepackage{amsthm}
\usepackage{eurosym}
\usepackage{graphicx}
\usepackage{epsfig}
\usepackage{dsfont}
\usepackage[nocompress, space]{cite}

\usepackage[displaymath,mathlines]{lineno}

\allowdisplaybreaks

\usepackage{hyperref}
\usepackage{pdfpages}

\usepackage{ifthen}
\usepackage[dvipsnames]{xcolor}
\usepackage[draft]{changes}
\definechangesauthor[name={Diogo}, color=blue]{DG}
\definechangesauthor[name={Reverse}, color=red]{R}
\makeindex

\newcommand{\essinf}{\operatorname{essinf}}
\renewcommand{\div}{\operatorname{div}}

\usepackage{mathtools}

\newcommand{\Rr}{{\mathbb{R}}}

\newcommand{\Nn}{{\mathbb{N}}}

\newcommand{\Tt}{{\mathbb{T}}}

\newcommand{\tiu}{{\widetilde{u}}}
\newcommand{\tim}{{\widetilde{m}}}
\newcommand{\mv}{{m_\sigma}}
\newcommand{\uv}{{u_\sigma}}

\newcommand{\Ii}{{\mathcal{I}}}

\def\dx{{\rm d}x}

\def\dt{{\rm d}t}
\def\geq{\geqslant}

\numberwithin{equation}{section}

\newtheoremstyle{thmlemcorr}{10pt}{10pt}{\itshape}{}{\bfseries}{.}{10pt}{{\thmname{#1}\thmnumber{
#2}\thmnote{ (#3)}}}
\newtheoremstyle{thmlemcorr*}{10pt}{10pt}{\itshape}{}{\bfseries}{.}\newline{{\thmname{#1}\thmnumber{
#2}\thmnote{ (#3)}}}
\newtheoremstyle{defi}{10pt}{10pt}{\itshape}{}{\bfseries}{.}{10pt}{{\thmname{#1}\thmnumber{
#2}\thmnote{ (#3)}}}
\newtheoremstyle{remexample}{10pt}{10pt}{}{}{\bfseries}{.}{10pt}{{\thmname{#1}\thmnumber{
#2}\thmnote{ (#3)}}}
\newtheoremstyle{ass}{10pt}{10pt}{}{}{\bfseries}{.}{10pt}{{\thmname{#1}\thmnumber{
A#2}\thmnote{ (#3)}}}

\theoremstyle{thmlemcorr}
\newtheorem{theorem}{Theorem}
\numberwithin{theorem}{section}
\newtheorem{lemma}[theorem]{Lemma}

\theoremstyle{thmlemcorr*}
\newtheorem{theorem*}{Theorem}
\newtheorem{lemma*}[theorem]{Lemma}
\newtheorem{corollary*}[theorem]{Corollary}
\newtheorem{proposition*}[theorem]{Proposition}
\newtheorem{problem*}[theorem]{Problem}
\newtheorem{conjecture*}[theorem]{Conjecture}

\theoremstyle{defi}
\newtheorem{definition}[theorem]{Definition}
\newtheorem{hyp}{Assumption}

\newtheorem{problem}{Problem}

\theoremstyle{remexample}
\newtheorem{remark}[theorem]{Remark}
\newtheorem{example}[theorem]{Example}

\newtheorem{cor}[theorem]{Corollary}

\theoremstyle{ass}

\begin{document}

\title{Weak-Strong Uniqueness for Second-Order Mean-Field Games}

\author[R. Ferreira]{Rita Ferreira}
\address[R. Ferreira]{
        King Abdullah University of Science and Technology (KAUST),
        CEMSE Division, Thuwal 23955-6900, Saudi Arabia.}
\email{rita.ferreira@kaust.edu.sa}

\author[D.  Gomes]{Diogo A. Gomes}
\address[D. A. Gomes]{
        King Abdullah University of Science and Technology (KAUST), CEMSE Division, Thuwal 23955-6900, Saudi Arabia.}
\email{diogo.gomes@kaust.edu.sa}

\author[B.  Majrashi]{Bashayer H. Majrashi}
\address[B.  Majrashi]{
        King Abdullah University of Science and Technology (KAUST),
        CEMSE Division, Thuwal 23955-6900, Saudi Arabia.}
\email{bashayer.majrashi@kaust.edu.sa} 

\begin{abstract}
We extend the weak–strong uniqueness principle for mean-field game (MFG) systems to a broad class of second-order stationary and time-dependent problems. Under standard monotonicity, growth, and coercivity assumptions on the Hamiltonian, and relying strictly on the integrability exponents guaranteed by the existing theory for monotone MFG systems, we show that any weak solution must coincide with a given strong solution. Our analysis covers models with spatially dependent scalar diffusion coefficients, using monotonicity arguments and a coefficient-adapted mollification strategy to manage the variable diffusion terms. We extend this strategy to establish weak–strong uniqueness in the corresponding second-order, initial-terminal, time-dependent setting. 
Finally, to address the critical quadratic growth regime, we derive a new a priori second-order estimate for a stationary MFG system with logarithmic coupling. This estimate provides quantitative bounds on the solution in terms of the data, and yields weak--strong uniqueness in the range where the improved integrability yields $L^2$ control of the density.
Since numerical and approximation methods for MFGs naturally yield weak solutions in the monotonicity sense, whereas strong solutions are known to exist in many settings, our results identify any weak limit produced by such methods with the strong solution whenever one exists.
\end{abstract}

\maketitle

\section{Introduction}

Mean-field games (MFGs) provide a framework for describing and analyzing the limiting behavior of differential games with a large number of identical agents. 
These models were introduced independently in \cite{ll1, ll2, lasryMeanFieldGames2007} and in \cite{huangLargePopulationStochastic2006}. For comprehensive surveys on PDE methods, modeling, and regularity, see~\cite{cardaliaguetIntroductionMeanField2020,GS,GPV,MR3691806}.
We consider the following model. Each player controls a state variable evolving on the $d$-dimensional torus $\Tt^d$ and seeks to minimize a running cost that depends on its own trajectory and the probability distribution of the population. The resulting equilibrium is characterized by a coupled system comprising
a Hamilton–Jacobi equation for the value function $u$ and a Fokker–Planck (or continuity) equation for the density $m$.  While existence theory is well developed (as detailed in Sections~\ref{Stationary-FW} and \ref{Evolutionary-FW}), 
establishing uniqueness for low-regularity solutions remains one of the main difficulties.
Here, we establish a weak-strong uniqueness principle, which asserts that a weak solution within our functional framework must coincide with a strong solution whenever the latter exists. While this has been established for first-order models~\cite{ferreiraWeakstrongUniquenessSolutions2025}, extending it to second-order systems presents a technical difficulty: standard regularization techniques do not commute with spatially dependent diffusion operators. We resolve this by introducing a coefficient-adapted mollification strategy. This allows us to prove weak–strong uniqueness for both stationary and time-dependent second-order MFGs with monotone couplings and general scalar diffusion on \(\Tt^d\) and \(\Tt^d\times(0,T)\), respectively, as we now explain. Beyond monotonicity and growth, the uniqueness theorems also require a strictly positive strong density and the Hamiltonian-compatibility assumptions collected in Section~\ref{secAs}.

\subsection{Stationary Framework}\label{Stationary-FW}
We consider the following second-order MFG with a scalar diffusion coefficient. 
\begin{problem}\label{Sproblem}
Let $\mathbb{E}$ be either $\Rr^+$ or $\Rr^+_0$.
   Consider a \( C^2 \) Hamiltonian \( H :\Tt^d \times \mathbb{R}^d \times \mathbb{E} \rightarrow \mathbb{R} \) and let $a\in C^2(\Tt^d)$ be such that $\min_{\Tt^d} a >0$. We seek pairs
   $(u,m)$, with
   \(
u: \Tt^d \rightarrow \Rr\) and  \(m: \Tt^d \rightarrow \mathbb{E} \)
satisfying the following second-order stationary MFG on the \( d \)-dimensional torus $\Tt^d$: 
\begin{equation}\label{SMFG2}
    \begin{cases}
        -u + a(x) \Delta u - H(x, Du, m) = 0 \\
        m -\Delta(a(x) m) -\div (m D_p H(x, Du, m)) = 1.
\end{cases}
\end{equation} 
\end{problem}

Integrating the second equation over $\Tt^d$ shows that $m$ is a probability measure, using the standard normalization $\int_{\Tt^d}\dx=1$.
Example~\ref{eg:typicalH} in Section~\ref{secAs} provides typical Hamiltonians for the preceding problem: \eqref{H:power-coupling} for the case $\mathbb{E}=\Rr^+_0$, and \eqref{H:congestion} for $\mathbb{E}=\Rr^+$. Strong and weak solutions to Problem~\ref{Sproblem} are formalized in Definitions~\ref{def:strongsol-S} and~\ref{def:weaksol}, respectively, and the standing assumptions on the Hamiltonian, diffusion coefficient, and integrability exponents are collected in Section~\ref{secAs}. Our main uniqueness result for this stationary setting is Theorem~\ref{Th:uniq-stat}, which asserts that under Assumptions~\ref{2.1}--\ref{2.13}, any weak solution coincides with a given strong solution.

\begin{remark}\label{rmk:onsign}
Note that Equation \eqref{SMFG2} is written with the sign reversed relative to the standard viscosity formulation
for Hamilton–Jacobi (HJ) equations, 
\[
u - a(x)\,\Delta u + H(x,Du,m)=0,
\] 
so that the operator $A$ in Definition~\ref{def:weaksol} is monotone in the sense required by our weak formulation.
\end{remark}

\begin{remark}
Although Problem~\ref{Sproblem} and its evolutionary counterpart (Problem~\ref{Tproblem} below) are formulated for general non-negative densities ($\mathbb{E}=\mathbb{R}^+_0$), our weak–strong uniqueness analysis requires the strong solution to be strictly positive (see Assumption~\ref{2.2}). This uniform lower bound is essential to ensure the coercivity of the linearized operator (see Step~2 in the proof of Theorem~\ref{Th:uniq-stat}). In the absence of such positivity, uniqueness may fail; for instance, \cite{Gomes2016b} provides explicit examples of non-uniqueness
and the formation of vacuum regions in one-dimensional first-order MFG models. The strict positivity of $m$ for regular enough solutions to \eqref{SMFG2} is expected, in view of the following formal argument. Suppose $\min m = m(x_0) = 0$. Then $Dm(x_0) = 0$ and
$\Delta m(x_0) \geq 0$. Since $m(x_0) = 0$ and $Dm(x_0) = 0$,
both $\Delta(a(x)m)|_{x_0} = a(x_0)\Delta m(x_0)$ and
$\div(m D_p H)|_{x_0} = 0$, so the left-hand side of the
second equation in \eqref{SMFG2} evaluates to
$-a(x_0)\Delta m(x_0) \leq 0$, contradicting the
right-hand side being $1$. 
\end{remark}

We now define the reference notion of strong solution.
\begin{definition}\label{def:strongsol-S}
    Given \( r_1, \gamma_1 \geq 1 \), a pair \( (m, u) \in W^{1,r_1}(\Tt^d) \times W^{1,\gamma_1}(\Tt^d) \) with \( m \in \mathbb{E}\) a.e.~in \( \Tt^d \) is a strong solution to Problem~\ref{Sproblem} if \eqref{SMFG2} holds in \( \mathcal{D}'(\Tt^d) \).
\end{definition}

We assume the existence of strong solutions, which is known in a number of examples that satisfy our assumptions. Prior results for classical solutions initially focused on discounted problems with quadratic Hamiltonians and logarithmic couplings, particularly in low dimensions~\cite{GM}. This approach was subsequently expanded to encompass quasi-variational structures~\cite{GPM1} and extended mean-field games, where costs depend on the velocity field, via the continuity method~\cite{GPatVrt}. Further generalizations addressed broader classes of Hamiltonians~\cite{PV15} and models with spatially varying diffusion coefficients~\cite{BFl}, the latter being directly relevant to the scalar diffusion
$ a(x)$ considered in Problem~\ref{Sproblem}. In the context of congestion models, existence has been established through the continuation method and a priori estimates for both sub-quadratic~\cite{EvGom} and quadratic~\cite{GMit} Hamiltonians; furthermore, for the quadratic case with power-like interactions, existence has been established in the sense of entropy solutions using truncation methods~\cite{bocorsporr}. Variational techniques have also been utilized to address hard congestion in a broad class of MFGs~\cite{AMFS}.

We define weak solutions as follows.
\begin{definition}\label{def:weaksol}
The operator $A$ associated with Problem~\ref{Sproblem}
    is defined, for
    $(\eta, v) \in C^\infty(\Tt^d; \mathbb{E}) \times C^\infty(\Tt^d)$,
    by
\[
A \begin{bmatrix} \eta \\ v \end{bmatrix} = \begin{bmatrix} -v + a(x)\Delta v - H(x, Dv, \eta) \\ (\eta-1) - \Delta(a(x) \eta) - \div (\eta D_p H(x, Dv, \eta)) \end{bmatrix}.
\]
    Given $r,\gamma\ge1$, a pair \( (\tim, \tiu) \in L^r(\Tt^d) \times W^{1,\gamma}(\Tt^d) \) with $\tim\ge0$ a.e.~in \(\Tt^d\) is a weak solution to Problem \ref{Sproblem} if
    \begin{align}
    \label{weaksol}
    \left\langle \begin{bmatrix} \eta \\ v \end{bmatrix} - \begin{bmatrix} \tim \\ \tiu \end{bmatrix}, A \begin{bmatrix} \eta \\ v \end{bmatrix} \right\rangle &=
         \int_{\Tt^d} (\eta - \tilde{m})\big(-v +a(x)\Delta v- H(x, Dv, \eta)\big) \,\dx \nonumber\\
        &\quad+\int_{\Tt^d} (v - \tilde{u})\big(\eta-1 -\Delta( a(x)\eta)- \div (\eta D_p H(x, Dv, \eta) )\big) \,\dx \ge 0
    \end{align}
    for all \( (\eta, v) \in C^\infty(\Tt^d; \mathbb{E}) \times C^\infty(\Tt^d) \). Here, $\langle\cdot,\cdot\rangle$ denotes the duality pairing between $C^\infty(\Tt^d) \times C^\infty(\Tt^d)$ and its dual. For distributions induced by functions, this pairing is given by
\begin{equation}
    \label{dualstat}
    \left\langle \begin{bmatrix} \eta \\ v \end{bmatrix}, \begin{bmatrix} \tilde{\eta} \\ \tilde{v} \end{bmatrix} \right\rangle = \int_{\Tt^d} (\eta \tilde{\eta} + v \tilde{v}) \,\dx.
\end{equation}
\end{definition}

Under the Lasry--Lions monotonicity condition,
the structural conditions formalized in
Assumption~\ref{2.13},
the operator $A$ is monotone:
$\langle w_1 - w_2, Aw_1 - Aw_2 \rangle \geq 0$ for all admissible
pairs $w_1, w_2$.

The existence of weak solutions in the sense of Definition~\ref{def:weaksol} has been established through a variational-inequality framework. This approach relies on a monotonicity-preserving regularization combined with Minty's method to recover the solution as the regularization parameter vanishes. Specifically, this framework has been applied in~\cite{FG2}
to first- and second-order stationary monotone systems in the periodic setting,  with and without congestion, encompassing Problem~\ref{Sproblem},  and later 
extended 
to first-order systems with Dirichlet boundary conditions in~\cite{FGT1}.  More recently, this variational approach was extended to Banach spaces, allowing for the treatment of first-order problems with and without congestion~\cite{ferreiraSolvingMeanFieldGames2025}.
By shifting the analysis from Hilbert spaces to Banach spaces (such as $L^p$ or $W^{1,p}$), the authors avoid the need for high-order smoothing, can treat both standard power-growth Hamiltonians and models featuring congestion, and unify earlier existence results.

\begin{remark}[Terminology]
Throughout the paper, we use the term {strong solution} for the Sobolev/distributional notion formalized in Definition~\ref{def:strongsol-S}. This corresponds to what is often called a “weak solution” in MFG theory outside the context of monotone operators.
We reserve the term {weak solution} for the lower-regularity solutions obtained via the monotonicity/Minty framework formalized in Definition~\ref{def:weaksol}.
\end{remark}

\subsection{Time-dependent Framework}\label{Evolutionary-FW}
    Turning to the evolutionary setting, we consider the following time-dependent mean-field game system:
\begin{problem}\label{Tproblem}
Let $\mathbb{E}$ be either $\Rr^+$ or $\Rr^+_0$.
    Consider a \( C^2 \) Hamiltonian \( H :\Tt^d \times \mathbb{R}^d \times \mathbb{E} \rightarrow \mathbb{R} \) and let $a\in C^2(\Tt^d)$ be such that $\min_{\Tt^d} a >0$. Fix a terminal condition $u_T:\Tt^d\rightarrow \Rr$ and an initial probability distribution $m_0:\Tt^d\rightarrow \mathbb{R}^+$, assumed smooth to simplify the approximation argument, specifically, in Section~\ref{FBM} to construct smooth approximants preserving the trace constraints.  We seek pairs
   $(u,m)$, with $u\colon \Tt^d\times (0,T)\rightarrow \Rr$ and $m\colon \Tt^d\times (0,T)\rightarrow \mathbb{E}$
such that the following second-order time-dependent MFG holds
on $\Tt^d\times (0,T)$:
\begin{equation}\label{TMFG}
    \begin{cases}
        u_t+a(x)\Delta u - H(x, Du, m) = 0 \\
        m_t  -\Delta(a(x) m) - \div (m D_p H(x, Du, m)) = 0,
\end{cases}
\end{equation}
 subject to the initial and terminal conditions
\begin{equation}\label{TC}
   \begin{cases}
        u(x,T) = u_T(x), \\
        m(x,0) = m_0(x).
    \end{cases}
\end{equation}
\end{problem}

Strong and weak solutions to Problem~\ref{Tproblem} are formalized in Definitions~\ref{def:strongsol-T} and~\ref{Tweak}, respectively, with initial-terminal conditions \eqref{TC} imposed through the boundary constraints \eqref{TestB.C.} on the test functions. The corresponding weak--strong uniqueness result is Theorem~\ref{Th:uniq-time}, established under the same set of assumptions as in the stationary case.

\begin{remark}
\label{R4}
Similarly to the stationary case mentioned in Remark~\ref{rmk:onsign},  we adopt a sign convention for the Hamilton–Jacobi (HJ) equation that differs from the standard backward formulation:
\[
- u_t - a(x)\,\Delta u + H(x,Du,m)=0
\]
to align \eqref{TMFG} with the monotone-operator framework. Specifically, we reverse the signs in \eqref{TMFG}
so that the operator $B$ defined in Definition~\ref{Tweak} satisfies the monotonicity condition used throughout this paper.
\end{remark}

Strong solutions in the evolutionary setting are defined as follows.
\begin{definition}\label{def:strongsol-T}
    Given \( r_1, \gamma_1 \geq 1 \), a pair \( (m, u) \in W^{1,r_1}(\Tt^d\times(0,T)) \times W^{1,\gamma_1}(\Tt^d\times(0,T)) \) with \( m \in \mathbb{E}\) a.e.~in \( \Tt^d\times(0,T)\) 
    is a strong solution to Problem~\ref{Tproblem} if \eqref{TMFG} holds in \( \mathcal{D}'(\Tt^d\times(0,T)) \) and \eqref{TC} holds in the sense of traces.
\end{definition}

The existence of such strong solutions is well understood. Classical (smooth) solutions were established in early work for time-dependent models with nonlocal and smoothing couplings~\cite{ll2}. Later, 
existence was established for local couplings in the subquadratic \cite{gomesTimeDependentMeanFieldGames2015} and superquadratic \cite{gomesTimedependentMeanfieldGames2016a} cases using the nonlinear adjoint method to obtain Lipschitz estimates for the value function.
 These techniques were further refined to handle models with
logarithmic nonlinearities \cite{Gomes2015b}, where establishing
lower bounds on the density $m$ requires delicate a priori estimates.
 Existence results have also been established for other notions of
strong solutions. For instance, models with degenerate diffusion
(including the first-order limit) and local couplings were treated
via dual optimal-control problems and variational methods
in~\cite{cgbt}. Meanwhile, monotone MFGs with congestion, where agent mobility is restricted by high population density, were treated via a priori estimates and compactness in~\cite{Achdou2016}. Moreover, \cite{porretta2} establishes the existence and uniqueness
of strong solutions for locally coupled second-order MFGs by
developing a weak theory that characterizes solutions of the
Fokker--Planck equation as renormalized solutions under minimal
assumptions on the optimal drift.

Analogous to Definition~\ref{def:weaksol}, we define weak solutions for the evolutionary case using the space-time duality pairing.

\begin{definition}\label{Tweak}
     The operator $B$ associated with Problem~\ref{Tproblem}
      is defined, for
      $(\eta, v) \in C^\infty(\Tt^d\times[0,T]; \mathbb{E})
      \times C^\infty(\Tt^d\times[0,T])$, by
      \[
B \begin{bmatrix} \eta \\ v \end{bmatrix} = \begin{bmatrix} v_t+a(x)\Delta v - H(x, Dv, \eta) \\
\eta_t-\Delta (a(x)\eta)  - \div (\eta D_p H(x, Dv, \eta)) \end{bmatrix}.
\]
Given \( r,\gamma \ge 1 \), a pair \( (\tim, \tiu) \in L^r(\Tt^d\times(0,T)) \times L^\gamma ((0,T); W^{1,\gamma}(\Tt^d)) \) with \( \tim \ge 0 \) a.e.~in \( \Tt^d\times(0,T) \) is a weak solution to Problem~\ref{Tproblem} if
    \begin{align}\label{weaksol2}
    \left\langle \begin{bmatrix} \eta \\ v \end{bmatrix} - \begin{bmatrix} \tim \\ \tiu \end{bmatrix}, B \begin{bmatrix} \eta \\ v \end{bmatrix} \right\rangle&= \int_0^T\!\!\!\int_{\Tt^d}
        (\eta - \tilde{m})\big(v_t + a(x)\Delta v - H(x, Dv, \eta)\big)\,\dx\dt \nonumber\\
    &\quad
    + \int_0^T\!\!\int_{\Tt^d}
        (v - \tilde{u})\big(\eta_t - \Delta(a(x)\eta) - \div(\eta D_p H(x, Dv, \eta))\big)\,\dx\dt \geq 0
    \end{align}
    for all \( (\eta, v) \in C^\infty(\Tt^d\times[0,T]; \mathbb{E}) \times C^\infty(\Tt^d\times[0,T])\) satisfying 
    \begin{equation}\label{TestB.C.}
    \begin{cases}
       \eta(\cdot, 0) = m_0(\cdot),\\  v(\cdot, T) = u_T(\cdot).
    \end{cases}
    \end{equation}
Here, $\langle\cdot,\cdot\rangle$ denotes the duality pairing given by
    \[
    \left\langle \begin{bmatrix} \eta \\ v \end{bmatrix}, \begin{bmatrix} \tilde{\eta} \\ \tilde{v} \end{bmatrix} \right\rangle = \int_0^T\!\!\!\int_{\Tt^d} (\eta \tilde{\eta} + v \tilde{v}) \,\dx\dt.
    \]
\end{definition}

As in the stationary case, the operator $B$ is monotone with respect
to the space-time pairing above under the structural conditions on $H$
formalized in Assumption~\ref{2.13}.

The existence of weak solutions in the sense of Definition~\ref{Tweak} is established in \cite{FGUtimeMFG2026}  using monotonicity methods on Banach spaces, extending the results in \cite{ferreiraSolvingMeanFieldGames2025} concerning stationary MFGs. This approach generalizes the results in \cite{FeGoTa21}, where existence was proven under the additional normalization condition $\int_{\Tt^d} \eta \,\dx = 1$ for the test functions. Under this condition, the work in \cite{FeGoTa21} develops a comprehensive theory for a broad class of time-dependent MFGs, including models with degenerate diffusion and local or nonlocal couplings. 

\subsection{Main Results}
Uniqueness across a broad class of MFGs follows from the
Lasry–Lions monotonicity condition
\cite{ll1, ll2, lasryMeanFieldGames2007}, which requires the MFG
system to define a monotone operator and the solutions to possess
sufficient integrability. A distinct notion of monotonicity, displacement monotonicity along
optimal transport geodesics, has also been developed to establish
uniqueness for broader classes of MFGs
\cite{meszarosMeanFieldGames2024, graberMonotonicityConditionsMean2023}.
While these methods have proven effective for strong solutions in
diverse settings, for example in first-order models with Neumann
boundary conditions~\cite{gomes2023parabolic}, standard arguments
cannot be directly applied to the weak solutions considered here
(Definitions~\ref{def:weaksol} and~\ref{Tweak}) due to their limited
regularity.
To bridge this gap, we extend the weak–strong uniqueness
principle of \cite{ferreiraWeakstrongUniquenessSolutions2025} to
second-order stationary and time-dependent MFGs. By integrating a
coefficient-adapted mollification scheme into the monotonicity
argument, we allow the duality inequality for weak solutions to imply
the variational identity for strong solutions, thereby proving that
$u=\tilde u$ and $m=\tilde m$ whenever a strong solution exists.

Beyond its intrinsic theoretical interest, the weak–strong uniqueness principle has direct consequences for the convergence of numerical methods. In many settings — particularly parabolic problems — the existence of strong (or classical) solutions is well established (see the references above). On the other hand, numerical schemes based on monotone discretizations or variational regularizations naturally produce approximations whose limits, when they exist, satisfy the weak formulation in the sense of Definitions~\ref{def:weaksol} and~\ref{Tweak}. Our results close this gap: any weak limit produced by such a method must coincide with the strong solution, thereby providing a unified convergence guarantee for a broad class of numerical approaches to second-order MFGs.
Our main contributions can be summarized as follows.

\paragraph{\bf Weak–strong uniqueness for second‐order stationary MFGs.} We show that under natural integrability, monotonicity, and coercivity conditions, together with the Hamiltonian-compatibility assumptions in Assumptions~\ref{2.4} and~\ref{2.10}, the weak solution (Definition~\ref{def:weaksol}) to a second-order stationary MFG with spatially varying diffusion coincides with the corresponding strong solution.
This result generalizes the weak--strong uniqueness principle introduced for first-order stationary models in \cite{ferreiraWeakstrongUniquenessSolutions2025}. The primary technical difficulty in the second-order setting arises from the fact that standard regularization techniques do not commute with spatially dependent diffusion operators; we resolve this by introducing a coefficient-adapted double-mollification argument.
   \begin{theorem}\label{Th:uniq-stat}
    Let \( (\tilde{m}, \tilde{u}) \in L^r(\Tt^d) \times W^{1,\gamma}(\Tt^d) \) be a weak solution (cf.~Definition~\ref{def:weaksol}) and \( (m, u) \in W^{1,r_1}(\Tt^d) \times W^{1,\gamma_1}(\Tt^d) \) be a strong solution (cf.~Definition~\ref{def:strongsol-S}) to Problem~\ref{Sproblem}. Suppose that Assumptions~\ref{2.1}, \ref{2.2}, \ref{2.4}, \ref{2.9}, \ref{2.10}, \ref{2.12}, and \ref{2.13} hold.
    Then, 
    \[
    u = \tilde{u} \quad \text{and} \quad m = \tilde{m}.
    \]
\end{theorem}
 \paragraph {\bf Weak–strong uniqueness for second-order time‐dependent MFGs.} Using a
 similar argument, we obtain a weak-strong uniqueness result for second-order evolution problems. To handle the time-dependent boundary conditions and the parabolic structure, the proof utilizes one-sided forward- and backward-in-time mollifiers.
    \begin{theorem}\label{Th:uniq-time}
    Let \( (\tilde{m}, \tilde{u}) \in L^r(\Tt^d\times (0,T))\times L^\gamma\left( (0,T); W^{1,\gamma}(\Tt^d)\right) \) be a weak solution (cf.~Definition~\ref{Tweak}) and \( (m, u) \in {W^{1,r_1}}(\Tt^d\times (0,T)) \times W^{1,\gamma_1}(\Tt^d\times (0,T)) \) be a strong solution (cf.~Definition~\ref{def:strongsol-T}) to Problem~\ref{Tproblem}. Suppose that Assumptions~\ref{2.1}, \ref{2.2}, \ref{2.4}, \ref{2.9}, \ref{2.10}, \ref{2.12}, and \ref{2.13} hold. 
    Then, 
    \[
    u = \tilde{u} \quad \text{and} \quad m = \tilde{m}.
    \]
\end{theorem}

\paragraph{\bf Improved regularity for weak solutions of the quadratic 
Hamiltonian stationary mean-field games.} The monotonicity argument used 
in~\cite{ferreiraWeakstrongUniquenessSolutions2025} and adapted here for 
second-order MFGs cannot be directly extended to MFGs with quadratic growth 
due to the lack of sufficient integrability for the density $\tilde{m}$ of weak solutions. 
In 
Section~\ref{sec:quad-reg}, we establish 

improved regularity results for weak solutions in a specific subclass of
Problem~\ref{Sproblem} featuring quadratic growth and logarithmic coupling
(Theorem~\ref{Th:IR}).
These regularity results restore the integrability required by the linearization argument, thereby enabling the proof of weak-strong uniqueness for these MFGs
(Theorem~\ref{Th:uniq-Q}).
In particular, combining Theorem~\ref{Th:IR} with Theorem~\ref{Th:uniq-Q} yields the logarithmic quadratic weak--strong uniqueness result in the range where $\tilde m\in L^2$, e.g.\ for $d\le 4$.
The remainder of the paper is organized as follows. Section~\ref{secAs} 
introduces the notation and standing assumptions used throughout the work. 
Section~\ref{sect:proofuniqstat} establishes the weak-strong uniqueness 
result for the stationary second-order MFG with scalar diffusion via a 
linearization technique and coefficient-adapted mollifiers. 
Section~\ref{sect:proofuniqTime} extends this argument to the 
time-dependent problem using forward and backward one-sided time mollifiers. 
Section~\ref{sect:quad-uniq} then addresses the stationary second-order 
MFGs with quadratic Hamiltonians. For these, the standard monotonicity 
argument cannot be directly applied. Accordingly, we first establish improved regularity for weak solutions 
of Problem~\ref{Sproblem} in the special case of logarithmic coupling, and then 
extend the analysis to prove weak–strong uniqueness for quadratic
Hamiltonians with monotone couplings, provided the weak solution
possesses the improved integrability established for logarithmic and
power-type nonlinearities.
Finally, in 
Section~\ref{sect:quad-conv}, we prove the strong convergence of monotone 
regularizations, completing the analysis of the quadratic regime.

\section{Preliminaries}\label{secAs}

\subsection{Notation}
 
We use the standard notation for Sobolev and Lebesgue spaces; namely, $W^{k,r}(\Tt^d)$ and $L^r(\Tt^d)$, respectively, where $r\in[1,+\infty]$. Given any such $r$, its conjugate exponent is denoted by $r'$, and is given by $\tfrac{1}{r} + \tfrac{1}{r'} = 1$. If $r\in [1,d)$, we denote by $r^*$  its Sobolev exponent, which is given by $r^*=\frac{dr}{d-r}$. Constants denoted by the same letter $C$ may change from line to line but depend only on fixed structural parameters. 
In particular, for $d>2$, we denote by $2^*$ the Sobolev exponent
\[
2^*=\frac{2d}{d-2}.
\]
When $d\le 2$, we use instead the Sobolev embeddings
\[
W^{1,2}(\Tt^d)\hookrightarrow L^q(\Tt^d) \qquad\text{for every finite } q\ge 1.
\]
We normalize the torus measure so that $\int_{\Tt^d}\dx=1$.
\subsection{Main Assumptions}

To facilitate a unified presentation, we introduce the following set:
\[
\Omega=\begin{cases}
\Tt^d                           & \text{in the stationary setting (Problem~\ref{Sproblem})},\\[0.5ex]
\Tt^d \times (0,T)              & \text{in the evolutionary setting (Problem~\ref{Tproblem})}.
\end{cases}
\]

As specified in Definitions~\ref{def:weaksol} and~\ref{Tweak}, weak solutions $(\tilde{m}, \tilde{u})$ belong to the spaces $L^r(\Tt^d) \times W^{1,\gamma}(\Tt^d)$  and $L^r(\Tt^d\times(0,T)) \times L^\gamma ((0,T); W^{1,\gamma}(\Tt^d))$, respectively, with $r,\gamma \geq 1$, consistent with the established existence theory
for stationary systems on the torus \cite{FG2}, models with Dirichlet boundary conditions \cite{FGT1}, and time-dependent problems through space-time regularization \cite{FeGoTa21}.

The assumptions below impose no additional regularity on the weak solution; rather, they require the strong solution to possess sufficient regularity to support the linearization argument. We categorize these assumptions into three groups: \textit{regularity requirements} (Assumptions \ref{2.1}--\ref{2.4}), \textit{integrability/growth conditions} (Assumptions \ref{2.9}--\ref{2.12}), and \textit{linearized coercivity} (Assumption \ref{2.13}).
The first assumption ensures that the strong solution space embeds into the weak solution space.

\begin{hyp}\label{2.1}
In the definition of weak and strong solutions, the exponent pairs \( (r, \gamma) \) and \( (r_1, \gamma_1) \)  satisfy \( r_1 \ge r > 1 \) and \( \gamma_1 \ge \gamma > 1 \). 
\end{hyp}
Next, we assume a uniform lower bound on the strong density, $m$. This bound allows us to consider convenient admissible perturbations of a strong solution and ensures coercivity at a key step  of our uniqueness results' proof. Without such a bound, 
uniqueness may fail; see the example in \cite{Gomes2016b} where $u$ is not unique in the regions where $m$ vanishes.

\begin{hyp}\label{2.2}
Suppose that $(m,u)\in W^{1,r_1}(\Omega)\times W^{1,\gamma_1}(\Omega)$ is a strong solution to Problem~\ref{Sproblem} (or Problem~\ref{Tproblem}). Then, there exists a constant $c_0>0$ such that
\begin{align*}
    \operatorname*{ess\,inf}_{y\in\Omega} m(y)\ge c_0.
\end{align*}
\end{hyp}

We further require sufficient regularity to allow testing with Sobolev-level perturbations in a neighborhood of the strong solution. This enables us to expand the class of admissible test functions for expressions involving differences between strong and weak solutions.

\begin{hyp}\label{2.4}
    Consider a strong solution \( (m, u)\in {W^{1,r_1}}(\Omega)\times
    W^{1,\gamma_1}(\Omega) \) to Problem~\ref{Sproblem} (or
    Problem~\ref{Tproblem}). There exists \( \varepsilon_0 > 0 \)
    such that for every
    \( (\eta, v) \in {W^{1,r_1}}(\Omega)\times
    W^{1,\gamma_1}(\Omega)\) satisfying
    \begin{equation}\label{eq:neigh}
    \| (m, u) - (\eta, v) \|_{{W^{1,r_1}}(\Omega)\times
    W^{1,\gamma_1}(\Omega) } \leq \varepsilon_0,
    \end{equation}
    \( \essinf_\Omega \eta > 0 \), and, in the time-dependent
    setting, \eqref{TestB.C.} in the sense of traces, we have
    \( H(\cdot, Dv, \eta) \in L^{r'}(\Omega) \) and
    \( \eta D_p H(\cdot, Dv, \eta) \in L^{\gamma'}(\Omega) \).
    Furthermore, there exists a sequence
    \( (\eta_k, v_k)_{k \in \mathbb{N}} \subset
    C^\infty(\overline\Omega; \Rr^+) \times
    C^\infty(\overline\Omega) \) such that:
    \begin{enumerate}
        \item \( \eta_k \to \eta \) in
              \( {W^{1,r_1}}(\Omega) \);
        \item \( v_k \to v \) in
              \( W^{1,\gamma_1}(\Omega) \), and in the following
              stronger topology if $v$ belongs to the
              corresponding space:
              \begin{itemize}
                  \item in $W^{2,r'}(\Tt^d)$ in the stationary
                        setting;
                  \item in $W^{1,r'}(0,T; L^{r'}(\Tt^d)) \cap
                        L^{r'}(0,T; W^{2,r'}(\Tt^d))$ in the
                        time-dependent setting;
              \end{itemize}
        \item \( H(\cdot, Dv_k, \eta_k) \to H(\cdot, Dv, \eta) \)
              in \( L^{r'}(\Omega) \);
        \item \( \eta_k D_p H(\cdot, Dv_k, \eta_k) \to
              \eta D_p H(\cdot, Dv, \eta) \) in
              \( L^{\gamma'}(\Omega) \);
        \item in the time-dependent setting,
              $\eta_k(\cdot,0) = m_0$ and $v_k(\cdot,T) = u_T$.
    \end{enumerate}
\end{hyp}

 \begin{remark}\label{2.5}
      Under the previous assumption, we have that \( H(\cdot, Du, m) \in L^{r'}(\Omega) \) and  \( m |D_p H(\cdot, Du, m)| \in L^{\gamma'}(\Omega) \). Accordingly, the  strong solution \((m, u)\in {W^{1,r_1}}(\Omega)\times W^{1,\gamma_1}(\Omega)  \) enjoys the following improved regularity: 
      \begin{itemize}
          \item In the stationary setting, the first equation in \eqref{SMFG2} is a uniformly elliptic PDE. By elliptic regularity theory (see \cite[Sec.~6.3]{evansPartialDifferentialEquations2010}), we have \[u \in W^{2,r'}(\Tt^d).\]
          
          \item For the time-dependent counterpart, the first equation in \eqref{TMFG} is a parabolic equation. By the $L^{r'}$-maximal regularity of the heat operator (see \cite[Theorem~5.4]{robinson2016N} for $d=3$; also see \cite[Sec.~7.1]{evansPartialDifferentialEquations2010}), we obtain
          \[
          u \in W^{1,r'}(0,T; L^{r'}(\Tt^d)) \cap L^{r'}(0,T; W^{2,r'}(\Tt^d)).
          \]
      \end{itemize}

\end{remark}

  In \cite{ferreiraWeakstrongUniquenessSolutions2025}, the authors establish the validity of Assumption~\ref{2.4} in the stationary setting for two typical classes of Hamiltonians (see Example~\ref{eg:typicalH}) using standard spatial mollification. However, extending this to the evolutionary framework presents a challenge, as standard space-time mollification may violate the boundary conditions specified in the last condition of Assumption~\ref{2.4}. Therefore,  inspired by the global approximation techniques in \cite[Chapter 5]{evansPartialDifferentialEquations2010}, we iteratively construct approximations specifically designed to ensure compatibility with the initial and terminal data.  The explicit construction of the test functions in Assumption~\ref{2.4} relies on a sequence of extensions, time-shifts, and mollifications; we present this construction in Section~\ref{FBM}.

In the following assumption, we impose convergence of the first and second derivatives of $H$ along smooth perturbations of a strong solution, used in our linearization argument within the weak-strong uniqueness proof.
\begin{hyp}\label{2.10}
    Let $(m,u)\in{W^{1,r_1}}(\Omega)\times W^{1,\gamma_1}(\Omega) $ be a strong solution to 
Problem~\ref{Sproblem} (or Problem~\ref{Tproblem})
and let $c_{0}$ and $\varepsilon_{0}$ be the positive constants specified by Assumptions~\ref{2.2} and~\ref{2.4}, respectively.
 For all \( (\bar{\eta}, \bar{v}) \in C^2(\overline \Omega) \times C^2(\overline \Omega) \), satisfying \(\bar\eta(\cdot,0)=0\) and \(\bar v(\cdot,T)=0\)  if $\Omega=\Tt^d\times (0,T)$, there exists \( 0 < \bar{\varepsilon} \leq  \varepsilon_0 \) such that, for \[ \bar{m}_\varepsilon = m + \varepsilon \bar{\eta} \quad \text{ and } \quad \bar{u}_\varepsilon = u + \varepsilon \bar{v} \] with \( |\varepsilon| \leq  \bar{\varepsilon} \), the following convergences hold as \( \varepsilon \to 0 \):
    \begin{enumerate}
        \item \( D_p H(\cdot, D\bar{u}_\varepsilon, \bar{m}_\varepsilon) \to D_p H(\cdot, Du, m) \) in \( L^{r'}(\Omega) \) and \( L^{\gamma'}(\Omega) \);
        \item \( D_m H(\cdot, D\bar{u}_\varepsilon, \bar{m}_\varepsilon) \to D_m H(\cdot, Du, m) \) in \( L^{r'}(\Omega) \);
        \item \( \bar{m}_\varepsilon D^2_{pp} H(\cdot, D\bar{u}_\varepsilon, \bar{m}_\varepsilon) \to m D^2_{pp} H(\cdot, Du, m) \) in \( L^{\gamma'}(\Omega) \);
        \item \( \bar{m}_\varepsilon D^2_{pm} H(\cdot, D\bar{u}_\varepsilon, \bar{m}_\varepsilon) \to m D^2_{pm} H(\cdot, Du, m) \) in \( L^{\gamma'}(\Omega) \).
    \end{enumerate}
\end{hyp}  

\begin{example}\label{eg:typicalH} In the examples below, the exponent $\gamma>1$ is the same exponent as the one appearing in the Sobolev space $W^{1,\gamma}$ used throughout the paper; we keep the same notation to emphasize this correspondence. Both Assumptions~\ref{2.4} and~\ref{2.10} are essentially intrinsic assumptions on the Hamiltonian.
As shown in \cite{ferreiraWeakstrongUniquenessSolutions2025}, together with the discussion in Section~\ref{FBM}, these conditions are satisfied by several standard classes of MFGs. 
Admissible Hamiltonians  include:
    \begin{itemize}
        \item The power-type Hamiltonian, with monotone coupling in \(m\), 
\begin{equation}\label{H:power-coupling}
H(x,p,m)\;=\;\frac{(1+|p|^{2})^{\gamma/2}}{\gamma}\;-\;g(m),
\qquad \gamma>1.
\end{equation}
 A typical choice
is $g(m) = m^\beta$ with $\beta > 0$, in which case the
integrability exponent for the density is $r = \beta + 1$ and
the assumptions are satisfied for a wide range of $\gamma$ and
$\beta$.
\item The congestion-type Hamiltonian
\begin{equation}\label{H:congestion}
H(x,p,m)\;=\;\frac{(1+|p|^{2})^{\gamma/2}}{\gamma\, m^{\alpha}},
\qquad \gamma>1,\ \alpha>0,
\end{equation}
   for a wide range of parameters $\gamma$ and  $\alpha$.
  \end{itemize}
\end{example}

Next, we impose compatibility conditions on the exponents to guarantee the integrability of the various integral densities appearing in our analysis.

\begin{hyp}\label{2.9}
 If the exponent $\gamma$ in the weak solutions satisfies \(\gamma < d\), we assume the inequality
\(
r' \leq  \gamma^*
\).
\end{hyp}

\begin{hyp}\label{2.12}
    Let the exponent pairs \((r,\gamma)\) and \((r_{1},\gamma_{1})\) satisfy the conditions of  Assumption \ref{2.1}. Assume further that
\begin{equation}\label{superQ}
     \frac{2}{r} + \frac{1}{\gamma} < 1 \quad \text{and} \quad \frac{1}{r} + \frac{2}{\gamma} < 1.
\end{equation}
    Define the exponents  \( q_1, q_2, q_3, \) and \( q_4 \) by
    \[
    \frac{1}{r} + \frac{1}{\gamma} + \frac{1}{q_1} = 1, \quad \frac{2}{r} + \frac{1}{q_2} = 1, \quad \frac{2}{r} + \frac{1}{\gamma} + \frac{1}{q_3} = 1, \quad \text{and} \quad \frac{1}{r} + \frac{2}{\gamma} + \frac{1}{q_4} = 1.
    \]
    For a strong solution \( (m, u) \in {W^{1,r_1}}(\Omega)\times W^{1,\gamma_1}(\Omega)  \), the following integrability properties hold:
    \begin{enumerate}
        \item \( D_p H(\cdot, Du, m) \in L^{q_1}(\Omega) \) and \( D_m H(\cdot, Du, m) \in L^{q_2}(\Omega) \);
        \item \( D^2_{pp} H(\cdot, Du, m) \in L^{q_4}(\Omega) \) and \( D^2_{pm} H(\cdot, Du, m) \in L^{q_3}(\Omega) \).
    \end{enumerate}
\end{hyp}
\begin{remark}
Inequality \eqref{superQ} implies that the second-order mean-field games in Problems~\ref{Sproblem} and~\ref{Tproblem}, when treated under the general framework of this section, are restricted to Hamiltonians with super-quadratic growth. The critical quadratic growth case ($r=1, \gamma=2$) requires specific handling and is treated separately in Section~\ref{sect:quad-uniq} and~\ref{sect:quad-conv}. This technical limitation is a standard feature of our technique,
 as seen in the first-order counterparts addressed in \cite{ferreiraWeakstrongUniquenessSolutions2025}.
\end{remark}
The following assumption ensures that the quadratic form generated by the linearization argument aforementioned is strictly positive, which is essential to deduce uniqueness.

\begin{hyp}\label{2.13}
    Suppose that $(m,u)\in{W^{1,r_1}}(\Omega)\times W^{1,\gamma_1}(\Omega) $ is a strong solution to Problem~\ref{Sproblem} (or Problem~\ref{Tproblem}). There exists a strictly positive function \( \omega : \Omega \rightarrow (0, \infty) \) such that
    \[
    \begin{bmatrix}
        m D^2_{pp} H(\cdot, Du, m) & \frac{1}{2} m D^2_{pm} H(\cdot, Du, m) \\
        \frac{1}{2} m D^2_{pm} H(\cdot, Du, m) & -D_m H(\cdot, Du, m)
    \end{bmatrix}
    \ge \omega I \quad \text{a.e. in } \Omega .
    \]
\end{hyp}

\subsection{Approximation}
\label{mol}

To overcome the lack of regularity to handle the diffusion and time-dependent boundary terms in our analysis, we employ a mollification argument so that all pairings are admissible and the problematic contributions vanish prior to a limit passage, yielding uniqueness in the limit. We specify next the mollifiers that we consider and their main properties.

\subsubsection{Stationary Mollifiers}
\label{R8}
To use the standard properties of mollifiers in our setting, we identify any function, $w$, in the $d$-dimensional torus $\Tt^d$ with a $Q$-periodic function, $w_\#$, on the Euclidean space $\Rr^d$, where $Q=(-\tfrac12,\tfrac12)^d$, and vice versa. In particular, $w\in C(\Tt^d)$ (respectively, $w\in L^r(\Tt^d)$) if, and only if, $w_\#\in C(\Rr^d)$ (respectively, $w_\# \in L^r_{\rm loc}(\Rr^d)$) and $w_\#$  is $Q$-periodic.

    Let $\theta\in C_c^\infty(\Rr^d)$ be a nonnegative smooth  kernel, supported in \(B(0, 1)\), with $\int_{\Rr^d}\theta\,\dx=1$ and $\theta(x)=\theta(-x)$ for all $x\in\Rr^d$. For  $x\in\Rr^d$ and ${ \delta}\in (0,\tfrac12)$, we consider the standard mollifier
$\theta_{ \delta}(x)=\tfrac{1}{{ \delta}^{d}}\theta(\tfrac{x}{{ \delta}})$.
Given a $Q$-periodic function $w_\# \in L^r_{\rm loc}(\Rr^d)$,  we define
\begin{align*}
    ( \theta_{ \delta} \ast w_\# )(x) = \int_{\Rr^d} \theta_{ \delta}(x-y) w_\#(y)\, dy,\quad x\in\Rr^d.
\end{align*}
A simple change of coordinates shows that $\theta_{ \delta} \ast w_\#$ is $Q$-periodic. Moreover, $(\theta_{ \delta} \ast w_\#) \in C^\infty(\Rr^d)$. Furthermore, 
$D(\theta_{ \delta} \ast w_\#) = \theta_{ \delta} \ast Dw_\#$ if $w_\# \in W^{1,r}_{\rm loc}(\Rr^d)$, and
\begin{enumerate}
        \item $(\theta_{ \delta} \ast w_\#) \to w_\#$ in $L^r(Q)$, \quad (respectively, in $W^{1,r}(Q)$ if $w_\# \in W^{1,r}_{\rm loc}(\Rr^d)$),
        
        \item $\theta_{ \delta} \ast(\theta_{ \delta} \ast w_\#) \to w_\#$ in  $L^r(Q)$ \quad (respectively, in $W^{1,r}(Q)$ if $w_\# \in W^{1,r}_{\rm loc}(\Rr^d)$).
\end{enumerate}        
Finally, the  symmetry of the mollifiers yields
\begin{equation}\label{eq:propmoll}
    \int_{Q}  \bar w_\#(x) (\theta_{ \delta} \ast  w_\#)(x) \,\dx = \int_{Q}  (\theta_{ \delta} \ast \bar  w_\#)(x)  w_\#(x) \,\dx
\end{equation}
for any $Q$-periodic functions  $w_\#$, $\bar  w_\# \in L^1_{\rm  loc}(\Rr^d)$.

Let \( a : \Tt^d \to \mathbb{R}^+ \) be as in the statement of Problem~\ref{Sproblem}, and consider the differences     \((u - \tiu)\in W^{1,{\gamma}}(\Tt^d)\) and \((m - \tim)\in L^r(\Tt^d)\) between strong and weak solutions. Denoting by $a_\# \in C^2(\Rr^d)$, \((u_\# - \tiu_\#)\in W^{1,\gamma}_{\rm loc}(\Rr^d)\) and \((m_\# - \tim_\#)\in L^r_{\rm loc}(\Rr^d)\) the corresponding $Q$-periodic functions in $\Rr^d$, we can define the double mollifications
    \[
    {\eta}_{ \delta}(x) = \frac{1}{a_\#(x)}\big(\theta_{ \delta} \ast \theta_{ \delta} \ast\big(a_\#(\cdot) (m_\# - \tim_\#)\big)\big)(x), \quad 
    {v}_{ \delta}(x) = \big(\theta_{ \delta} \ast \theta_{ \delta} \ast (u_\# - \tiu_\#)\big)(x)
    \]  
for $x\in\Rr^d$. 
Using the discussion above, we have  $ \eta_{ \delta} \in C^2(\Rr^d)$ and   $v_{ \delta} \in C^\infty(\Rr^d)$, both $Q$-periodic functions,   and
\begin{align*}
    {\eta}_{ \delta} \to m_\# - \tim_\# \quad \text{in } L^r(Q), \quad {v}_{ \delta} \to u_\# - \tiu_\# \quad \text{in } W^{1,\gamma}(Q).
\end{align*}
The $C^2$ regularity of $\eta_{ \delta}$ suffices, since the linearized identity \eqref{E6onQ} is tested only against $C^2$ pairs.

\subsubsection{Evolutionary Mollifiers}
\label{emol}
  To handle the boundary terms in the time-dependent case, we start by considering the  zero extensions of the difference between a strong and a weak solution as follows:
\begin{equation}
\begin{aligned}\label{eq:defext0}
\bar m(x,t) &=
\begin{cases}
m(x,t)-\tilde m(x,t), & 0 \leq  t \leq  T,\\
0,                       &\text{otherwise},
\end{cases}\\
\bar u(x,t) &=
\begin{cases}
u(x,t)-\tilde u(x,t), & 0 \leq t \leq T,\\
0,                       & \text{otherwise}.
\end{cases}
\end{aligned}
\end{equation}

Let \(\zeta\in C_c^\infty(-1,1)\) be a  nonnegative and symmetric smooth kernel, with \(\int_{\Rr}\zeta\,\dt=1\). Given $\rho>0$, we define the standard mollifier, $\zeta_\rho\in C_c^\infty(-\rho,\rho)$, by  
\(\zeta_\rho(\cdot)=\tfrac{1}{\rho}\zeta(\tfrac{\cdot}{\rho})\). We further consider the following one-sided time shifted mollifiers (see Figure~\ref{fig:one-sided-mollifiers}): 

\begin{itemize}
    \item one-sided, forward-shifted mollifier, $\psi_{\rho}\in C_c^\infty(0,\rho)$, defined by
\(
\psi_{\rho}(\cdot)=\zeta_{\tfrac{\rho}{2}}\!\left(\cdot-\tfrac{\rho}{2}\right);\)
\item one-sided, backward-shifted mollifier, $\phi_{\rho}\in C_c^\infty(-\rho,0)$, defined by 
\(
\phi_{\rho}(\cdot)=\zeta_{\tfrac{\rho}{2}}\!\left(\cdot+\tfrac{\rho}{2}\right).
\)
\end{itemize}
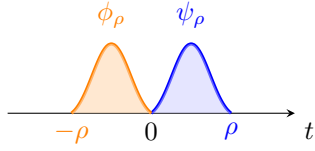
\begin{figure}[htbp]
  \centering
\begin{tikzpicture}[>=stealth,scale=1.05]
  \draw[->] (-1.8,0) -- (1.8,0) node[below right] {$t$};
  \node at (0,0) [below] {$0$};


  \draw[domain=0:1,smooth,variable=\x,blue,very thick]
    plot({\x},{3.5*(\x)*(1-\x)*exp(-6*(\x-0.5)^2)});
  \fill[blue!20,opacity=0.5]
    (0,0) -- plot[domain=0:1,smooth,variable=\x] ({\x},{3.5*(\x)*(1-\x)*exp(-6*(\x-0.5)^2)}) -- (1,0) -- cycle;
  \node[blue] at (1,0) [below] {$\rho$};
  \node[blue] at (0.5,1.25) {$\psi_\rho$};

  \draw[domain=-1:0,smooth,variable=\x,orange,very thick]
    plot({\x},{3.5*(1+\x)*(-\x)*exp(-6*(\x+0.5)^2)});
  \fill[orange!30,opacity=0.6]
    (-1,0) -- plot[domain=-1:0,smooth,variable=\x] ({\x},{3.5*(1+\x)*(-\x)*exp(-6*(\x+0.5)^2)}
) -- (0,0) -- cycle;
  \node[orange] at (-1,0) [below] {$-\rho$};
  \node[orange] at (-0.5,1.25) {$\phi_\rho$};
\end{tikzpicture}
  \caption{Typical profile of the one-sided time shifted mollifiers $\phi_\rho$ and $\psi_\rho$.}
  \label{fig:one-sided-mollifiers}
\end{figure}
Finally, we  define the time-reversal mollifiers, $\hat{\psi}_\rho$ and $\hat{\phi}_\rho$, by setting  
\[\hat{\psi}_\rho(t)=\psi_\rho(-t)=\phi_\rho(t), \qquad \hat{\phi}_\rho(t)=\phi_\rho(-t)=\psi_\rho(t), \quad t\in\Rr.\]
Then,  we have the  following adjointness property  for \(f,g\in L^1_{\mathrm{loc}}(\mathbb{R})\):
\[
\int_{\mathbb{R}} f(t)\,(\phi_\rho\ast_t g)(t)\,\dt
=\int_{\mathbb{R}} g(t)\,(\hat\phi_\rho\ast_t f)(t)\,\dt = \int_{\mathbb{R}} g(t)\,(\psi_\rho\ast_t f)(t)\,\dt,
\]
where $\ast_t$ denotes the convolution in time and, as usual,  is given by 
\[
(\phi_\rho \ast_t f)(t) = \int_\Rr \phi_\rho(t-s)f(s) \, ds.
\]

We use these mollifiers to define the sequences
\[\bar{\eta}_\rho=\psi_\rho\ast_t\psi_\rho\ast_t \bar{m},\quad \bar{v}_\rho=\phi_\rho\ast_t\phi_\rho\ast_t \bar{u}, \]
for which it can be checked that
\begin{enumerate}
        \item \(\bar{\eta}_\rho \to \bar{m} \quad \text{in } {L^r(\Tt^d\times (0,T))}\);
        \item \(\bar{v}_\rho \to \bar{u} \quad \text{in } {L^\gamma\left( \Tt^d\times (0,T)\right)}\);
        \item \(\bar\eta_\rho(\cdot,t)=0\) for \(t \leq 0\) and \(\bar v_\rho(\cdot,t)=0\) for \(t \ge T\);
        \item \((\psi_\rho\ast_t \bar{m})(\cdot,t) = 0\) for \(t \leq 0\) and \((\phi_\rho\ast_t \bar{u})(\cdot,t)=0\) for \(t \ge T\).
    \end{enumerate}

\subsubsection{Approximation in the evolutionary case}
\label{FBM}

We now verify the approximation component of Assumption~\ref{2.4} in the evolutionary
setting by explicitly constructing smooth approximants compatible with the initial and terminal data. 
    Let \( (\eta, v) \in {W^{1,r_1}}(\Omega)\times W^{1,\gamma_1}(\Omega)\) be a pair satisfying \eqref{TestB.C.}  in the sense of traces, such that \(  c_\eta=\essinf_\Omega \eta > 0 \). We define
    \begin{align}\label{eq:barc}
        \bar c = \min\Big\{ c_\eta, \min_{\Tt^d} m_0 \Big\}>0,
    \end{align}
noting that $\min_{\Tt^d} m_0 >0$ is guaranteed because $m_0$ is assumed to be a smooth positive function in Problem~\ref{Tproblem}.

\begin{enumerate}
        \item \textbf{Extensions:} We extend  \(\eta\) for \(t < 0\) by setting it equal to \(m_0\), denoting the extended function by $\bar \eta$, and we extend  \(v\) for \(t > T\) by setting it equal to \(u_T\), which we denote by $\bar v$. More precisely,
        \[
        \bar{\eta}(x,t) = \begin{cases}
           m_0(x), &\quad  t < 0, \\
           \eta(x,t), &\quad 0\leq t\leq T,
        \end{cases}
        \quad \text{and} \quad
        \bar{v}(x,t) = \begin{cases}
           v(x,t), &\quad 0\leq t\leq T, \\
           u_T(x), &\quad t > T.
        \end{cases}
        \]

        \item \textbf{Shift and Mollification:} For \(0<{  h}<\frac T2\), we define the time-shifted functions \(\bar{\eta}_{  h}\colon  \Tt^d \times (-\infty, T+{  h}] \to \mathbb{R}^+\)  and \(\bar{v}_{  h}\colon \Tt^d \times [-{  h}, \infty) \to \Rr\)  by
        \[
        \bar{\eta}_{  h}(x,t) = \bar{\eta}(x, t - {  h}) \quad \text{and} \quad \bar{v}_{  h}(x,t) = \bar{v}(x, t + {  h}).
        \]
        Next, for \(0<{ \lambda}<\frac{  h}2\),  we consider a standard space-time mollifier \({  \mu}_{ \lambda}\in C^\infty_c(\Rr^d \times \Rr)\)   supported in \(B_{ \lambda}(0)\subset \Rr^d \times \Rr\). 

        \item \textbf{Boundary Correction, while preserving Positivity:} Finally, we define 
        \begin{align*}
            &\bar \eta_{  h}^{ \lambda} (x,t) = ({  \mu}_{ \lambda} \ast_{x,t} (\bar{\eta}_{  h} - m_0))(x,t) + m_0(x),\\
            &\bar v_{  h}^{ \lambda} (x,t) = ({  \mu}_{ \lambda} \ast_{x,t} (\bar{v}_{  h} - u_T))(x,t) + u_T(x).
        \end{align*}

Given the smoothness of $m_0$ and $u_T$ assumed in the problem statement, we have $(\bar \eta_{  h}^{ \lambda}, \bar v_{  h}^{ \lambda}) \in C^\infty (\overline \Omega) 
\times C^\infty (\overline \Omega)$.

On the other hand, because ${  \mu}_{ \lambda}(\cdot, -s) = 0$ for $|s|>{ \lambda}$ and $(\bar{\eta}_{  h} - m_0)(\cdot, s) =0$ for $|s|\leq { \lambda}$, we have $\bar \eta_{  h}^{ \lambda} (x,0) = m_0(x)$. Similarly, ${  \mu}_{ \lambda}(\cdot, T -s) = 0$ for $|T-s|>{ \lambda}$ while $(\bar v_{  h} - u_T) (\cdot,s) =0 $ for $|T-s|\leq{ \lambda}$, yielding $\bar v_{  h}^{ \lambda} (x,T)= u_T(x)$.

Regarding the positivity of $\bar \eta_{  h}^{ \lambda}$, we observe that $\lim_{{ \lambda}\to0} \Vert m_0 - {  \mu}_{ \lambda} \ast_{x,t} m_0\Vert_{L^\infty(\Omega)} = 0$ because $m_0$ is, in particular, a continuous function. Thus, 
for all ${ \lambda}$ sufficiently small, we have $\Vert m_0 - {  \mu}_{ \lambda} \ast_{x,t} m_0\Vert_{L^\infty(\Omega)} < \frac{\bar c}2$, where $\bar c$ is given by \eqref{eq:barc}. Consequently, for all such 
${ \lambda}$, it follows that
\begin{align*}
    \bar \eta_{  h}^{ \lambda} = {  \mu}_{ \lambda} \ast_{x,t} \bar{\eta}_{  h} + m_0 - {  \mu}_{ \lambda} \ast_{x,t} m_0 > \bar c - \tfrac{\bar c}2 = \tfrac{\bar c}2.
\end{align*}
        
        \item \textbf{Convergence:} The properties of the shift and convolution in the second  point  above, together with the identities $\bar \eta = \eta$ and $\bar v = v$ in $\Omega$, yield
        \begin{align*}
            \lim_{{  h}\to0} \lim_{{ \lambda}\to0} \Vert \bar \eta_{  h}^{ \lambda} - \eta \Vert_{W^{1,r_1}(\Omega)} = 0 \quad \text{and}\quad \lim_{{  h}\to0} \lim_{{ \lambda}\to0} \Vert \bar v_{  h}^{ \lambda} - v \Vert_{W^{1,\gamma_1}(\Omega)} = 0.
        \end{align*}
 Then, a diagonal argument gives the sequence required for the approximation part of Assumption~\ref{2.4}.  
    \end{enumerate}

\section{Proof of Uniqueness for Second-Order Stationary MFG with Scalar Diffusion Coefficient}\label{sect:proofuniqstat}
In this section, we prove the weak–strong uniqueness of solutions to Problem~\ref{Sproblem} (Theorem~\ref{Th:uniq-stat}). 
This result extends prior work on first-order stationary MFGs \cite{ferreiraWeakstrongUniquenessSolutions2025} to the second-order setting with a spatially varying diffusion coefficient, $a(\cdot)$. We characterize the system \eqref{SMFG2} as a monotone operator and linearize the corresponding variational inequality at the strong solution. To handle the variable diffusion terms, we employ the coefficient-adapted mollification strategy detailed in Section~\ref{R8}, utilizing specific perturbations that ensure the precise cancellations required for the uniqueness proof.

\begin{proof}[Proof of Theorem~\ref{Th:uniq-stat}]
As in the theorem statement, let \( (\tilde{m}, \tilde{u}) \in L^r(\Tt^d) \times W^{1,\gamma}(\Tt^d) \) be a weak solution and \( (m, u) \in W^{1,r_1}(\Tt^d) \times W^{1,\gamma_1}(\Tt^d) \) be a strong solution to Problem~\ref{Sproblem}. By Remark~\ref{2.5}, we have $u\in W^{2,r'}(\Tt^d)$. Moreover, 
because Assumption~\ref{2.12} implies $\gamma/2>r'$, and Assumption~\ref{2.1} gives $\gamma_1\ge \gamma$, we conclude that
\begin{align*}
    u \in W^{1,\gamma_1}(\Tt^d) \cap W^{2,r'}(\Tt^d).
\end{align*}

We now proceed in three steps.

\textbf{Step 1 (Comparing the weak inequality and the strong identity for a common smooth pair).}
Fix $(\eta,v)\in C^\infty(\Tt^d; \mathbb{E}) \times C^\infty(\Tt^d)$. Starting from the weak solution (cf.~Definition~\ref{def:weaksol}), we  integrate by parts in the last two integral terms in \eqref{weaksol} to obtain
  \begin{align}\label{E4}
    &\int_{\Tt^d} (\eta - \tilde{m})(-v +a(x)\Delta v- H(x, Dv, \eta)) \,\dx \nonumber \\
    &\quad + \int_{\Tt^d} \left[(v - \tilde{u})(\eta-1)+ (Dv - D\tilde{u}) \cdot (D(a(x)\eta)+\eta D_p H(x, Dv, \eta))\right] \,\dx \ge 0.
\end{align}
    For the strong solution (cf.~Definition~\ref{def:strongsol-S}), we use Assumptions~\ref{2.1} and~\ref{2.9}, condition \eqref{superQ} in Assumptions~\ref{2.12}, and the elliptic regularity discussed in Remark~\ref{2.5} to use pairs in $L^r(\Tt^d)\times W^{1,\gamma}(\Tt^d)$ as test functions in the sense of distributions. This allows us to 
    test \eqref{SMFG2} against the pair $(\eta-\tim,\;v-\tiu)$, from which we obtain
\begin{align}
    &\int_{\Tt^d} (\eta - \tilde{m}) \big(-u + a(x) \Delta u - H(x, Du, m)\big) \,\dx \nonumber \\
    &\quad + \int_{\Tt^d} (v - \tilde{u}) (m -1)+ (Dv - D\tilde{u}) \cdot \big(m D_p H(x, Du, m) + D(a(x)m)\big) \,\dx = 0.\label{E5}
\end{align}

\textbf{Step 2 (Reducing regularity of the smooth pair \((\eta,\; v)\)).} 
Now, we show that \eqref{E4} and \eqref{E5} hold for test functions \( (\eta, v)\in W^{1,r_1}(\Tt^d;\Rr^+) \times  (W^{1,\gamma_1}(\Tt^d) \cap W^{2,r'}(\Tt^d)) \)  in a certain neighborhood of the strong solution.

Let $\varepsilon_0>0$ be given by Assumption~\ref{2.4}, and fix 
a pair \( (\eta, v)  \in W^{1,r_1}(\Tt^d; \mathbb{R}^+) \times  (W^{1,\gamma_1}(\Tt^d)  \cap \allowbreak W^{2,r'}(\Tt^d)) \), with $\essinf \eta>0$ and satisfying \eqref{eq:neigh}, and an approximating smooth sequence \( (\eta_k, v_k)_k \subset C^\infty(\Tt^d; \mathbb{R}^+) \times C^\infty(\Tt^d) \), as described  in Assumption~\ref{2.4}. Substituting this smooth sequence into \eqref{E4} and \eqref{E5} in place of $(\eta,v)$ and 
using Assumptions~\ref{2.1} and~\ref{2.9}, condition \eqref{superQ} in Assumptions~\ref{2.12}, and the convergence properties in Assumption~\ref{2.4} to pass to the limit in the resulting expressions, we conclude that  \eqref{E4} and \eqref{E5} hold for any \( (\eta, v) \in W^{1,r_1}(\Tt^d; \mathbb{R}^+) \times  (W^{1,\gamma_1}(\Tt^d) \cap W^{2,r'}(\Tt^d))  \) satisfying $\essinf \eta>0$ and \eqref{eq:neigh}.

In particular, for any such pair $(\eta,v)$, we obtain by subtracting \eqref{E5} from \eqref{E4} that
  \begin{align}\label{Subtracted}
    &\int_{\Tt^d} (\eta - \tilde{m})\big(-(v - u) +a(x)(\Delta v-\Delta u)- H(x, Dv, \eta) + H(x, Du, m)\big) \,\dx  \nonumber\\
    &\quad + \int_{\Tt^d} (v - \tilde{u})(\eta - m) \,\dx + \int_{\Tt^d} (Dv - D\tilde{u}) \cdot \big(D(a(x)\eta)-D(a(x)m)\big) \,\dx  \nonumber\\
    &\quad + \int_{\Tt^d} (Dv - D\tilde{u}) \cdot \big(\eta D_p H(x, Dv, \eta) - m D_p H(x, Du, m)\big) \,\dx \ge 0.
\end{align}

Let $\bar{\varepsilon}\in(0,\varepsilon_0]$ be given by Assumption~\ref{2.10}, and take \((\bar\eta,\bar v)\in C^2(\Tt^d)\times C^2(\Tt^d)\) arbitrarily. Then, using Assumption~\ref{2.2} and Remark~\ref{2.5}, we can find $\varepsilon_1\in (0,\bar\varepsilon]$ such that for all $\varepsilon\in [-\varepsilon_1,\varepsilon_1]$, the functions defined by \(v_\varepsilon=u+\varepsilon\bar v\) and \(\eta_\varepsilon=m+\varepsilon\bar\eta\) satisfy \( (\eta_\varepsilon, v_\varepsilon) \in W^{1,r_1}(\Tt^d; \mathbb{R}^+) \times  (W^{1,\gamma_1}(\Tt^d) \cap W^{2,r'}(\Tt^d))  \), 
\[
\essinf \eta_\varepsilon>0,\quad \text{and} \quad \|(m,u)-(\eta_\varepsilon,v_\varepsilon)\|_{W^{1,r_1}(\Tt^d)\times W^{1,\gamma_1}(\Tt^d)}\leq \varepsilon_0.
\]
In particular, for any such $\varepsilon \in [-\varepsilon_1,\varepsilon_1]$, we can use 
\eqref{Subtracted} with \( \eta = \eta_\varepsilon \) and \( v = v_\varepsilon\), from which we deduce that  
        \begin{align}
        \label{Ie}
    \Ii(\varepsilon) &= \int_{\Tt^d} (m - \tilde{m} + \varepsilon \bar{\eta}) \big( -\varepsilon \bar{v} +\varepsilon a(x)\Delta\bar{v} - H(x, Dv_\varepsilon, \eta_\varepsilon) + H(x, Du, m) \big)\,\dx \\\notag
    &\quad + \int_{\Tt^d} (u - \tilde{u} + \varepsilon \bar{v})(\varepsilon \bar{\eta}) + (Du - D\tilde{u} + \varepsilon D\bar{v}) \cdot \big(\varepsilon D( a(x) \bar{\eta}) \big)\,\dx \\\notag
    &\quad + \int_{\Tt^d} (Du - D\tilde{u} + \varepsilon D\bar{v}) \cdot \big( \eta_\varepsilon D_p H(x, Dv_\varepsilon, \eta_\varepsilon) - m D_p H(x, Du, m) \big)\,\dx \ge 0.
    \end{align}
Hence, $\Ii(0) = 0$ and $\Ii(\varepsilon) \geq 0$ for  $\varepsilon\in[-\varepsilon_1,\varepsilon_1]$; in other words, $\Ii$ attains a minimum at $\varepsilon=0$. Because $\Ii$ is differentiable in $(-\varepsilon_1,\varepsilon_1)$, we have $\Ii'(0) = 0$. Furthermore, by applying Remark~\ref{2.5} and Assumption \ref{2.10}, we may differentiate under the integral sign and then let $\varepsilon\to0^+$ to conclude that $\Ii'(0) = 0$ can be equivalently written as

    \begin{align}\label{E6}
       &\int_{\Tt^d} \Big[ (m - \tilde{m}) \left( - \bar{v} - D_p H(x, Du, m) \cdot D\bar{v} - D_m H(x, Du, m) \bar{\eta} \right) + (u - \tilde{u}) \bar{\eta}\Big]\,\dx \nonumber \\
        &\quad  + \int_{\Tt^d} 
    \big[(m-\tim)(a(x)\Delta\bar{v})+(Du-D\tiu)\cdot(D(a(x)\bar{\eta}))\big]\,\dx \nonumber \\
    &\quad + \int_{\Tt^d} (Du - D\tilde{u}) \cdot \big( \bar{\eta} D_p H(x, Du, m) + m D^2_{pp} H(x, Du, m) D\bar{v} \big)\,\dx \nonumber\\
    &\quad + \int_{\Tt^d} (Du - D\tilde{u}) \cdot \big(m D^2_{pm} H(x, Du, m) \bar{\eta} \big)\,\dx = 0.
    \end{align}
    \textbf{Step 3 (Main linearization argument).}
    If $a\equiv 0$ in \eqref{E6}, we recover the case treated in \cite{ferreiraWeakstrongUniquenessSolutions2025}. In this case, the linearity of $\Ii'(0)$ in $(\bar\eta, \bar v)$ and Assumption~\ref{2.12} would allow us to conclude that $\Ii'(0) = 0$ holds for all $\bar\eta \in L^r(\Tt^d)$ and $\bar v \in W^{1,\gamma}(\Tt^d)$, matching the regularity of the weak solution. We could then choose $(\bar\eta, \bar v)=(m-\tilde m, u - \tilde u)$, yielding a quadratic form where Assumption~\ref{2.13} ensures uniqueness. We note that the diffusion terms in \eqref{E6} formally cancel
upon choosing $(\bar\eta, \bar v) = (m - \tilde m, u - \tilde u)$,
since
\[
\int_{\Tt^d} \big[(m - \tilde m)(a(x)\Delta(u - \tilde u))
+ (Du - D\tilde u) \cdot D(a(x)(m - \tilde m))\big]\,\dx = 0
\]
by integration by parts. However, this computation requires
$(m - \tilde m) \in W^{1,\gamma'}(\Tt^d)$ and
$(u - \tilde u) \in W^{2,r'}(\Tt^d)$, which are not available at the
regularity level of the weak solution. The coefficient-adapted
mollification introduced below approximates precisely this
cancellation while remaining within the admissible regularity class.

To implement this strategy, we use the periodic extension and regularization discussed in Section~\ref{R8}. Specifically, identifying a function $w$ on $\Tt^d$ with a $Q$-periodic function $w_\#$ on $\Rr^d$, where $Q=(-\tfrac12,\tfrac12)^d$, \eqref{E6} is equivalent to saying that
\begin{align}\label{E6onQ}
       &\int_{Q} \Big[ (m_\# - \tilde{m}_\#) \big( - {v} - D_p H_\#(x, Du_\#, m_\#) \cdot D{v} - D_m H_\#(x, Du_\#, m_\#) {\eta} \big) + (u_\# - \tilde{u}_\#) {\eta}\Big]\,\dx \nonumber \\
        &\quad  + \int_{Q} 
    \big[(m_\#-\tim_\#)(a_\#(x)\Delta{v})+(Du_\#-D\tiu_\#)\cdot D(a_\#(x){\eta})\big]\,\dx \nonumber \\
    &\quad + \int_{Q} (Du_\# - D\tilde{u}_\#) \cdot \big( {\eta} D_p H_\#(x, Du_\#, m_\#) + m_\# D^2_{pp} H_\#(x, Du_\#, m_\#) D{v} \big)\,\dx \nonumber\\
    &\quad + \int_{Q} (Du_\# - D\tilde{u}_\#) \cdot \big(m_\# D^2_{pm} H_\#(x, Du_\#, m_\#) {\eta} \big)\,\dx = 0
    \end{align}
for all $Q$-periodic pairs \((\eta, v)\in C^2(\Rr^d)\times C^2(\Rr^d)\). Let ${  \delta} > 0$ be the mollification parameter. For $x\in\Rr^d$, we define
\begin{align}\label{eq:molletaepsi}
{\eta}_{  \delta} (x) = \frac{1}{a_\#(x)}\big(\theta_{  \delta}\ast\big(\theta_{  \delta}\ast(a_\#(m_\#-\tilde m_\#))\big)\big)(x)
\end{align}
and
\begin{align}\label{eq:mollvepsi}
{v}_{  \delta}(x) = \big(\theta_{  \delta}\ast\big(\theta_{  \delta}\ast(u_\#-\tilde u_\#)\big)\big) (x),
\end{align}
as in Section~\ref{R8}. Substituting $(\eta,v) = ({\eta}_{  \delta} ,{v}_{  \delta})$ into \eqref{E6onQ} and using the symmetry of the mollifiers (see \eqref{eq:propmoll}), the diffusion terms vanish identically for any ${  \delta}$:

 \begin{align*}
     &\int_{Q} \big[
    (m_\#-\tim_\#)(a_\#(x)\Delta{v}_{  \delta})+(Du_\#-D\tiu_\#)\cdot D(a_\#(x){\eta}_{  \delta})\big]\,\dx\\
    &\quad = \int_{Q} 
    (m_\#-\tim_\#)(a_\#(x)\Delta(\theta_{  \delta} \ast (\theta_{  \delta} \ast(u_\#-\tiu_\#))))\,\dx\\
    &\qquad +\int_{Q} (Du_\#-D\tiu_\#)\cdot D(\theta_{  \delta} \ast (\theta_{  \delta} \ast( a_\#(\cdot) (m_\#-\tim_\#)))) \,\dx\\
    &\quad = -\int_{Q} D(\theta_{  \delta} \ast (a_\#(\cdot) (m_\#-\tim_\#)) )\cdot D(\theta_{  \delta} \ast (u_\#-\tiu_\#)) \,\dx\\
    &\qquad +\int_{Q}D(\theta_{  \delta} \ast (a_\#(\cdot) (m_\#-\tim_\#))) \cdot D(\theta_{  \delta} \ast (u_\#-\tiu_\#)) \,\dx\\&\quad =0.
 \end{align*}
  Moreover, using the regularity of the mollifiers as seen in Section \ref{R8}, we can take the limit  \({  \delta} \to 0\) in the remaining integral terms, which, written on $\Tt^d$, yields
\begin{equation*}
    \begin{aligned}
     &\int_{\Tt^d}   \Big[ - D_m H(x, Du, m)(m - \tilde{m})^2 +(Du - D\tilde{u})^T  m D^2_{pp} H(x, Du, m) (Du - D\tilde{u})\\
     &\qquad +\, m D^2_{pm} H(x, Du, m) \cdot (Du - D\tilde{u})(m - \tilde{m}) \Big]\,\dx = 0.
    \end{aligned}
\end{equation*}
Applying Assumption~\ref{2.13} with $\xi=Du-D\tiu$ and $z=m-\tim$, the integrand is a quadratic form in $(Du - D\tiu, m - \tim)$ that is positive definite by Assumptions~\ref{2.13} and~\ref{2.2}, yielding $\omega(x)(|Du-D\tiu|^2+|m-\tim|^2)=0$ a.e.
Accordingly, we must have $m = \tim$ and $Du = D\tiu$ a.e. Finally, using these identities in \eqref{E6} yields $\int_{\Tt^d} (u- \tiu)\bar \eta \,\dx =0$ for all $\bar\eta\in C^2(\Tt^d)$, which concludes the proof.
\end{proof}

\section{Proof of Uniqueness for Second-Order Time-Dependent Mean-Field Games}\label{sect:proofuniqTime}

We now adapt the weak-strong uniqueness argument from the stationary case (Section~\ref{sect:proofuniqstat}) to the time-dependent setting. The core strategy of linearizing the variational inequality at the strong solution remains the same. The main difficulty imposed by the time-dependent framework, which is absent in the stationary problem, is the appearance of boundary terms in time  that arise from  integrating  time-dependent terms by parts.

We use the one-sided mollifiers which vanish at the time boundaries $t=0$ and $t=T$, see Section \ref{FBM}. These ensure that all problematic time-boundary terms are zero before passing to the limit (cf. Step~3 of the proof of Theorem~\ref{Th:uniq-stat}). This technique, combined with spatial mollifiers, allows us to replicate the stationary argument, isolate the coercive quadratic form, and conclude that the weak and strong solutions coincide.

\begin{proof}[Proof of Theorem~\ref{Th:uniq-time}]
As mentioned above, we proceed analogously to the proof of Theorem~\ref{Th:uniq-stat}; therefore, we detail only the main differences.

\textbf{Step 1 (Testing weak and strong solutions).} Take $(\eta, v)\in C^\infty(\Tt^d\times[0,T];\mathbb{E})\times C^\infty(\Tt^d\times[0,T])$ satisfying \eqref{TestB.C.}. 
Performing an integration by parts in the weak formulation \eqref{weaksol2} leads to
\begin{align}\label{E4t}
  &\int_0^T\!\!\!\int_{\Tt^d} (\eta-\tilde m)\big(v_t+a(x)\Delta v-H(x,Dv,\eta)\big)\,\dx\dt \nonumber \\
  &\quad +\int_0^T\!\!\!\int_{\Tt^d} \big[(v-\tilde u)\eta_t + (Dv-D\tilde u)\cdot\big(D(a(x)\eta)+\eta D_pH(x,Dv,\eta)\big)\big]\,\dx\dt \ge 0.
  \end{align}

Next, testing the strong solution against $(\eta-\tim,\;v-\tiu)$, which is an admissible pair by the regularity in Remark~\ref{2.5} and Assumption~\ref{2.9}, we obtain
\begin{align}\label{E5t}
  &\int_0^T\!\!\!\int_{\Tt^d} (\eta-\tilde m)\big(u_t+a(x)\Delta u-H(x,Du,m)\big)\,\dx\dt \nonumber \\
  &\quad +\int_0^T\!\!\!\int_{\Tt^d}\big[ (v-\tilde u)m_t + (Dv-D\tilde u)\cdot\big(D(a(x)m)+mD_pH(x,Du,m)\big)\big]\,\dx\dt =0.
  \end{align}

\textbf{Step 2 (Reducing regularity and linearization).}
Let $\varepsilon_0>0$ be given by Assumption~\ref{2.4}, and take $(\eta,v)\in X_1 \times X_2$, where
$X_1 = W^{1,r_1}(\Tt^d\times(0,T))$ and
$X_2 = W^{1,\gamma_1}(\Tt^d\times(0,T))
\cap W^{1,r'}(0,T; L^{r'}(\Tt^d))
\cap L^{r'}(0,T; W^{2,r'}(\Tt^d))$, with $\essinf\eta>0$, satisfying \eqref{TestB.C.} in the sense of traces and  \eqref{eq:neigh}. 

 Choosing a smooth approximation
as in Assumption~\ref{2.4}, and then passing to the limit using Assumptions~\ref{2.9} and \ref{2.1}, 
we conclude that \eqref{E4t} and \eqref{E5t} hold for every such Sobolev pair $(\eta,v)$.
 Subtracting the left-hand side of \eqref{E5t} from the left-hand side of \eqref{E4t} yields, for all such $(\eta,v)$,
\begin{align}
    &\int_0^T\!\!\!\int_{\Tt^d} (\eta - \tilde{m})\big((v - u)_t +a(x)(\Delta v-\Delta u)- H(x, Dv, \eta) + H(x, Du, m)\big) \,\dx\dt \nonumber  \\
    &\quad +\int_0^T \!\!\!\int_{\Tt^d} \big[ (v - \tilde{u})(\eta - m)_t+ (Dv - D\tilde{u}) \cdot \big(D(a(x)\eta)-D(a(x)m)\big)\big] \,\dx\dt\nonumber \\
    &\quad+\int_0^T\!\!\! \int_{\Tt^d} (Dv - D\tilde{u}) \cdot \big(\eta D_p H(x, Dv, \eta) - m D_p H(x, Du, m)\big) \,\dx\dt \ge 0. \label{eq:weaktimes}
\end{align}
Fix $\bar\eta,\bar v\in C^2(\Tt^d\times[0,T])$ satisfying \(\bar\eta(\cdot,0)=0\) and \(\bar v(\cdot,T)=0\), and let $\bar\varepsilon>0$ be as in Assumption~\ref{2.10}. For $\varepsilon\in (-\bar\varepsilon, \bar\varepsilon)$, define the perturbed pair
\[
(\eta_\varepsilon, v_\varepsilon)= (m+\varepsilon\bar\eta, u+\varepsilon\bar v).
\]
By Assumption~\ref{2.2}, we have for all  $|\varepsilon|$ small enough that $\eta_\varepsilon\ge0$, \(\eta_\varepsilon(\cdot,0)=m_0(\cdot)\) and \(v_\varepsilon(\cdot,T)=u_T(\cdot)\) in the sense of traces, and 
\[
\|(\eta_\varepsilon,v_\varepsilon)-(m,u)\|_{W^{1,r_1}(\Tt^d\times(0,T))\times W^{1,\gamma_1}(\Tt^d\times(0,T))}\le\varepsilon_0.
\]
We can then take $(\eta,v)=(\eta_\varepsilon,v_\varepsilon)$ in \eqref{eq:weaktimes}, from which we obtain 
\begin{align*}
    \mathcal I(\varepsilon)=&\int_0^T \!\!\!\int_{\Tt^d} (m-\tilde m+\varepsilon\bar\eta)\big(\varepsilon\bar v_t+\varepsilon a(x)\Delta\bar v - H(x,Dv_\varepsilon,\eta_\varepsilon)+H(x,Du,m)\big)\,\dx\dt\\
     &+\int_0^T \!\!\!\int_{\Tt^d}(u-\tilde u+\varepsilon\bar v)(\varepsilon\bar\eta)_t+(Du-D\tilde u+\varepsilon D\bar v)\cdot\big( \varepsilon  D(a(x)\bar\eta) \big)\,\dx\dt\\
    &+\int_0^T \!\!\!\int_{\Tt^d} (Du-D\tilde u+\varepsilon D\bar v)\cdot\big( \eta_\varepsilon D_pH(x,Dv_\varepsilon,\eta_\varepsilon)-mD_pH(x,Du,m)\big)\,\dx\dt \;\ge 0.
\end{align*}
Since $\Ii(\varepsilon) \ge 0$ for all $\varepsilon$ and $\Ii(0) = 0$, $\Ii$ attains a minimum at $\varepsilon=0$. Furthermore, because $\Ii$ is differentiable, we must have $\Ii'(0) = 0$. Next, we observe that Remark~\ref{2.5} and Assumption~\ref{2.10} allow us to differentiate under the integral, yielding
\begin{align}
    \Ii'(0)  &= I + II + \int_0^T\!\!\!\int_{\Tt^d}  (m- \tilde m) (-D\bar v \cdot D_p H(x,Du,m) -\bar \eta D_m H(x,Du,m))\,\dx\dt\nonumber\\
    &\quad + \int_0^T\!\!\!\int_{\Tt^d} \big[ (Du - D\tilde{u}) \cdot \big(\bar{\eta} D_p H(x, Du, m) + m D^2_{pp} H(x, Du, m) D\bar{v} \nonumber\\
    &\qquad\qquad\quad+\, m D^2_{pm} H(x, Du, m) \bar{\eta}\big)\big]\,\dx\,\dt,\label{eq:I'(0)}
\end{align}
where
\begin{align}
    &I= \int_0^T\!\!\!\int_{\Tt^d} \big [(m - \tilde{m})  \bar{v}_t  +(u-\tiu)\bar{\eta}_t\big] \,\dx\,\dt,\label{eq:I}\\
    &II=\int_0^T\!\!\! \int_{\Tt^d} \big[ (m - \tilde{m}) a(x) \Delta\bar{v} + (Du-D\tiu)\cdot  D(a(x)\bar{\eta})\big]\,\dx\dt.\label{eq:II}
\end{align}

\textbf{Step 3 (Main mollification argument).}
Following the strategy from Step~3 of Theorem~\ref{Th:uniq-stat}, we will choose $(\bar\eta,\bar v)$ to be specific mollified approximations in space and time of  $(m - \tilde m)$ and $(u - \tilde u)$. This construction ensures that the problematic terms, first $II$ and then $I$,  vanish before passing to the limit. 

For simplicity, we will drop the explicit distinction between functions on $\Tt^d$ and their $Q$-periodic identification on $\Rr^d$, as discussed in Section \ref{R8} and detailed in the proof of Theorem~\ref{Th:uniq-stat}.
Also, to differentiate  the spatial convolution from the temporal convolution (cf.~Section~\ref{mol}), we will use the notation $\ast_x$ and $\ast_t$, respectively. Recalling the zero extensions $\bar m$ and $\bar u$ in \eqref{eq:defext0} of $(m - \tilde m)$ and $(u - \tilde u)$, respectively,   we  define the following test functions:
\[
\bar\eta^{\rho,{  \delta}}=\psi_\rho\ast_t \big( \psi_\rho\ast_t(\tfrac{1}{a}\kappa_{  \delta}\ast_x(a\bar m))\big)\quad \text{and} \quad 
\bar v^{\rho,{  \delta}}=\phi_\rho\ast_t \big( \phi_\rho\ast_t(\kappa_{  \delta}\ast_x\bar u)\big),
\]
where $\kappa_{  \delta}=\theta_{  \delta}\ast_x\theta_{  \delta}$. We observe $\bar\eta^{\rho,{  \delta}}$, $\bar v^{\rho,{  \delta}} \in C^2(\overline \Omega)$ and,  as in point 3 of Section~\ref{FBM}, 
\(\bar\eta^{\rho,{  \delta}}(\cdot,t)=0\) for \(t \leq 0\) and \(\bar v^{\rho,{  \delta}}(\cdot,t)=0\) for \(t \ge T\).

Using integration by parts and the properties of the spatial and one-sided time mollifiers, we find that the variable diffusion cross-terms cancel exactly: the term $II$ in \eqref{eq:II} with $\bar \eta$ replaced by $\bar\eta^{\rho,{  \delta}}$ and $\bar v$ replaced by $\bar v^{\rho,{  \delta}}$ vanishes:
\begin{align*}
 &\int_0^T\!\!\! \int_{Q} \big[ (m - \tilde{m}) a(x) \Delta\bar{v}^{\rho,{  \delta}} + (Du-D\tiu)\cdot D(a(x)\bar{\eta}^{\rho,{  \delta}})\big]\,\dx\dt\\
 &\quad = \int_\Rr \int_{Q} \big[ \bar m a(x) \Delta\big(\phi_\rho\ast_t \big( \phi_\rho\ast_t(\kappa_{  \delta}\ast_x\bar u)\big)\big)  +D\bar u \cdot D\big(\psi_\rho\ast_t \big( \psi_\rho\ast_t(\kappa_{  \delta}\ast_x (a \bar m))\big)\big)\big]\,\dx\,\dt\\
    &\quad =\int_\Rr \int_{Q} \big[ \big(\hat{\phi}_\rho\ast_t (\hat{\phi}_\rho\ast_t (a(x)\bar m))\big)\Delta\big(
    \kappa_{  \delta}\ast_x\bar{u}\big)\\
    &\qquad\qquad\qquad +D\big(\theta_{  \delta}\ast_x \bar u \big) \cdot D\big(\theta_{  \delta}\ast_x\big(\psi_\rho\ast_t(\psi_\rho\ast_t(a\bar{m}))\big)\big)\big]\,\dx\dt\\
    &\quad =\int_\Rr \int_{Q}\big[-D\big(\theta_{  \delta}\ast_x\big(\hat{\phi}_\rho\ast_t (\hat{\phi}_\rho\ast_t (a\bar{m}))\big)\big)\cdot D\big(\theta_
    {  \delta}\ast_x\bar{u}\big)\\
    &\qquad\qquad\qquad +D\big(\theta_{  \delta}\ast_x \bar u \big)\cdot D\big(\theta_{  \delta}\ast_x\big(\psi_\rho\ast_t(\psi_\rho\ast_t(a\bar{m}))\big)\big)\big]\,\dx\dt\\
    &\quad =0
\end{align*}
for all $\rho, {  \delta}>0$, where the last identity follows from
$\hat{\phi}_\rho = \psi_\rho$ (cf.~Section~\ref{emol}).

Hence, passing to the limit as ${  \delta}\to0$ in the identity $0=\Ii'(0)$ with $\bar \eta$ replaced by $\bar\eta^{\rho,{  \delta}}$ and $\bar v$ replaced by $\bar v^{\rho,{  \delta}}$ in \eqref{eq:I'(0)}, \eqref{eq:I}, and \eqref{eq:II}, we conclude that
\begin{align}
    0 & = \int_0^T\!\!\!\int_{\Tt^d}  (m- \tilde m) (-D\bar v^{\rho} \cdot D_p H(x,Du,m) -\bar \eta^{\rho} D_m H(x,Du,m))\,\dx\dt\nonumber\\ 
    &\quad + \int_0^T\!\!\!\int_{Q} \Big[ (Du - D\tilde{u}) \cdot \big(\bar{\eta}^{\rho} D_p H(x, Du, m) + m D^2_{pp} H(x, Du, m) D\bar{v}^{\rho} \nonumber\\
    &\qquad\qquad\quad+ m D^2_{pm} H(x, Du, m) \bar{\eta}^{\rho}\big)\Big]\,\dx\,\dt\nonumber\\
    &\quad+\int_0^T\!\!\!\int_{Q} \big [(m - \tilde{m})  \bar{v}^{\rho}_t  +(u-\tiu)\bar{\eta}^{\rho}_t\big] \,\dx\,\dt,\label{eq:justrho}
\end{align}
where
\[
\bar\eta^{\rho}=\psi_\rho\ast_t ( \psi_\rho\ast_t\bar m)\quad \text{and} \quad 
\bar v^{\rho}=\phi_\rho\ast_t ( \phi_\rho\ast_t\bar u).
\]

Next, recalling the discussion in Section~\ref{FBM} once more, together with an integration by parts, we observe that the last integral term in \eqref{eq:justrho} vanishes:
\begin{align*}
    &\int_0^T\!\!\!\int_{Q} \big [(m - \tilde{m})  \bar{v}^{\rho}_t  +(u-\tiu)\bar{\eta}^{\rho}_t\big] \, dx dt\\
    &\quad = \int_\Rr\int_{Q} \Big[\bar m   \big(\phi_\rho\ast_t \big( \phi_\rho\ast_t\bar u\big)\big)_t  +\bar u \big(\psi_\rho\ast_t \big( \psi_\rho\ast_t\bar m\big)\big)_t \Big]\, dx dt\\
     &\quad= \int_\Rr\int_{Q} \Big[(\hat{\phi}_\rho\ast_t\bar m)  \left(\phi_\rho\ast_t \bar{u}\right)_t  +\big(\hat{\psi}_\rho\ast_t \bar u \big) \big(\psi_\rho\ast_t\bar{m}\big)_t\Big] \, dx dt \\
     &\quad= \int_0^T\int_{Q} \Big[(\hat{\phi}_\rho\ast_t\bar m)  \left(\phi_\rho\ast_t \bar{u}\right)_t  +\big(\hat{\psi}_\rho\ast_t \bar u \big) \big(\psi_\rho\ast_t\bar{m}\big)_t\Big] \, dx dt \\
     &\quad=\int_0^T \int_{Q} \Big[(\psi_\rho\ast_t\bar m)  (\phi_\rho\ast_t \bar{u})_t-(\phi_\rho\ast_t \bar u)_t (\psi_\rho\ast_t\bar{m}) \Big] dx dt \\
     &\qquad+\left[\int_{Q} \left(\phi_\rho\ast_t \bar u \right) \left(\psi_\rho\ast_t\bar{m}\right) dx\right]_0^T\\
     &\quad=0.
\end{align*}
Consequently, letting $\rho\to0$ in \eqref{eq:justrho} yields
\begin{align*}
    &\int_0^T\!\!\!\int_{\Tt^d} \big[ - D_m H(x, Du, m)(m - \tilde{m})^2 + (Du - D\tilde{u})^T  m D^2_{pp} H(x, Du, m)  (Du - D\tilde{u}) \\
    &\qquad\qquad\quad
+ m D^2_{pm} H(x, Du, m) \cdot (Du - D\tilde{u})(m - \tilde{m})\big] \,\dx\dt = 0,
\end{align*}
which together with Assumption~\ref{2.13} 
implies that $Du=D\tilde u$ and $m=\tilde m$ a.e.~in $\Tt^d\times (0,T)$. Substituting $Du=D\tilde u$ and $m=\tilde m$ into $\Ii'(0)=0$ gives
\[
\int_0^T\!\!\int_{\Tt^d} (u-\tilde u)\,\bar\eta_t\,\dx\dt=0
\]
for all $\bar\eta\in C^2(\Tt^d\times[0,T])$ with $\bar\eta(\cdot,0)=0$. Given any $\chi\in C^\infty(\Tt^d\times[0,T])$, set
\[
\bar\eta(x,t)=\int_0^t \chi(x,s)\,ds.
\]
Then $\bar\eta(\cdot,0)=0$ and $\bar\eta_t=\chi$, so
\[
\int_0^T\!\!\int_{\Tt^d} (u-\tilde u)\chi\,\dx\dt=0.
\]
Hence $u=\tilde u$ a.e.~in $\Tt^d\times(0,T)$.
\end{proof}

\begin{remark}
In the  $a=0$ case, Problem~\ref{Tproblem} reduces to a first-order system. For this case, the weak solutions defined by \eqref{Tweak} coincide with strong solutions $(m,u)$ belonging to the space $W^{1,r_1}((0,T);L^{r_1}(\Tt^d)) \times W^{1,\gamma}((0,T)\times \Tt^d)$.
\end{remark}

\section{Weak-Strong Uniqueness for Stationary Mean-Field Game Systems with Quadratic Hamiltonian}\label{sect:quad-uniq}

Now, we turn to the quadratic Hamiltonian case. The new ingredient relative to \cite{FG2} is the second-order estimate in Lemma~\ref{Lemma:SOE} for nonconstant diffusion and logarithmic coupling, which upgrades the regularity of weak solutions beyond the general monotonicity theory. This model presents specific challenges because the standard monotonicity machinery used in Sections~\ref{sect:proofuniqstat} and \ref{sect:proofuniqTime} (and in \cite{ferreiraWeakstrongUniquenessSolutions2025}) cannot be applied directly due to the growth of the Hamiltonian. The primary obstacle is the limited regularity of the weak solutions induced by monotonicity (in the sense of Definition~\ref{def:weaksol}) at this borderline growth regime. Specifically, the existence theory established in \cite{FG2} yields weak solutions that belong only to the space $L^1(\Tt^d ) \times W^{1,2}(\Tt^d )$. Under such limited regularity, Assumption~\ref{2.12} is violated because the integrability condition \eqref{superQ} fails to hold for the exponents $\gamma=2$ and $r=1$. To address this, we formulate the problem as follows.
\begin{problem}\label{Qproblem}
Let $\mathbb{E}$ be either $\Rr^+$ or $\Rr^+_0$.
    Let \( F : \Tt^d \times \mathbb{E} \rightarrow \Rr \) be \(C^\infty\) and strictly increasing in its second variable,  let $a\in C^2(\Tt^d)$ be such that $\min_{\Tt^d} a >0$, and  let $b\in C^1(\Tt^d;\Rr^d)$. 
    We seek pairs
   $(u,m)$, with 
    \(
        u:\Tt^d \to \Rr\) and \(m:\Tt^d \to \mathbb{E}
    \),
    such that the following  stationary second-order MFG system holds on \( \Tt^d \):
    \begin{equation}\label{QMFG}
        \begin{cases}
            -u + \div (a(x) D u ) -b(x)\cdot Du - \frac{1}{2}\,|Du|^2 + F(x,m) = 0\\
            m - \div (a(x) D m ) -\div (mb(x)) - \div (mDu ) = 1.
        \end{cases}
    \end{equation}
\end{problem}

\begin{remark}
    We observe that Problem~\ref{Qproblem} is an instance of Problem~\ref{Sproblem}  by taking $H(x,p,m) = \tfrac{|p|^2}{2} - Da(x)\cdot p +b(x)\cdot p - F(x,m)$ in the latter. We further note that \eqref{QMFG} allows us to go beyond  the simpler quadratic case in which the Hamilton--Jacobi equation is given by $-u + a(x) \Delta  u  - \frac{1}{2}\,|Du|^2 + F(x,m) = 0$, while being able to retrieve this case by taking $b=Da$.  
\end{remark}

In this section, we first improve the regularity of the weak solutions. While sufficient regularity is known for power-like couplings (\cite[Section~6]{FG2}), we extend this to logarithmic couplings below.
Then, in Subsection~\ref{sec6.2}, we establish weak–strong uniqueness for quadratic MFGs, relying on the existence of weak solutions with this improved regularity (covering both logarithmic and power-like couplings). Finally, Section \ref{sect:quad-conv} establishes the strong convergence of the regularization procedure introduced in Section \ref{sec:quad-reg}.

\subsection{Improved regularity for the logarithmic coupling}\label{sec:quad-reg}
In this section, we analyze a subclass of Problem~\ref{Sproblem} characterized by a quadratic Hamiltonian and logarithmic coupling. Our objective is to establish new a priori estimates and regularity results for the density $m$ and the value function $u$ that improve upon the existence regularity obtained in \cite{FG2}. While the variational-inequality methods in that work ensure the existence of weak solutions in  $L^1(\Tt^d ) \times W^{1,2}(\Tt^d)$, this regularity is insufficient for the linearization argument required for uniqueness at the quadratic growth threshold. To overcome this limitation, we focus our analysis on the following specific formulation.

\begin{problem}\label{Qproblem-log}
   Let $a, V\in C^\infty(\Tt^d)$ be such that $\min_{\Tt^d} a >0$, and  let $b\in C^\infty(\Tt^d;\Rr^d)$. We seek pairs
   $(u,m)$, with \(
u: \Tt^d \;\rightarrow\; \Rr\) and \(m: \Tt^d \;\rightarrow\;\Rr^+ \),
such that the following second-order stationary MFG holds on the \( d \)-dimensional torus $\Tt^d$:
\begin{equation}\label{QMFG-log}
\begin{cases}
    -u+\div (a(x)D u ) -b(x)\cdot Du - \frac{1}{2}|D u|^2 - V(x)= -\log(m)\\
m-\div(a(x)D m) -\div (mb(x)) -\div (mDu) = 1.
\end{cases}
\end{equation}
\end{problem}

In the context of strong solutions, certain quadratic MFGs were studied using the Hopf--Cole transformation in \cite{GeRC}, and later 
a similar technique was used in \cite{cirantGeneralizationHopfColeTransformation2015} and \cite{MR4132070}.
While Sobolev regularity has been established for first-order models~\cite{graber2017sobolev} and singular couplings have been explored in other contexts~\cite{2016arXiv161107187C}, those results typically apply to stronger notions of solutions. Proving such estimates specifically for weak solutions in the sense of monotonicity demands a tailored approach. Here, we derive higher-order spatial estimates to restore the necessary integrability.

Following the methodology established in \cite{FG2}, the analysis proceeds by introducing a regularized version of Problem~\ref{Qproblem-log}, constructed by perturbing \eqref{QMFG-log} with higher-order elliptic terms and a singular penalization function. This regularization maintains the monotone structure of the system while ensuring the existence of classical, smooth solutions. We subsequently derive first- and second-order a priori estimates, yielding quantitative bounds that remain uniform with respect to the regularization parameter $\sigma$. These estimates enable a passage to the limit via the Minty--Browder method \cite{Eva, KiSt00}. This allows us to construct weak solutions with significantly better regularity than the $L^1(\Tt^d) \times W^{1,2}(\Tt^d)$ solutions established in \cite{FG2}. 
The regularized problem reads:
\begin{equation}\label{Rmfg}
\begin{cases}
-u+\div(a(x)Du) -b(x)\cdot Du -\frac12|Du|^2-V(x)+\sigma(m+\Delta^{2k}m)+\beta_\sigma(m)=-\log(m)\\
m-\div(a(x)Dm) -\div (mb(x)) -\div(m Du)+\sigma(u+\Delta^{2k}u)=1,
\end{cases}
\end{equation}
where $a$, $V$, and $b$ are as in Problem~\ref{Qproblem-log}, $k\in\mathbb{N}$ is chosen so that $2k-4>\frac{d}{2}+1$, $\sigma\in (0,1)$, and $\beta_\sigma:(0,\infty)\to(-\infty,0]$ is a non-decreasing and $C^\infty$ penalization function that we detail next. For $q>d$, we set
\begin{align}
\beta_\sigma(s)=
\begin{cases}
- s^{-q} & \text{for } 0<s\le\sigma/2,\\
0        & \text{for } s>\sigma,
\end{cases}\label{eq:defbeta}
\end{align}
while on the transition interval, $(\sigma/2,\sigma]$, $\beta_\sigma(\cdot)$ is defined as any smooth monotone interpolation connecting the two regimes in \eqref{eq:defbeta}, serving as a penalization that forces $m$ to be strictly positive while preserving the monotonicity of the system.  We denote by $A_\sigma$ the operator associated with this regularized system \eqref{Rmfg}, given by 
\begin{equation}\label{Rmfg2}
A_\sigma \begin{bmatrix}\eta\\ v\end{bmatrix}=\begin{bmatrix}-v+\div(a(x)Dv) -b(x)\cdot Dv - \frac12|Dv|^2 - V(x) + \sigma(\eta+\Delta^{2k}\eta) + \beta_\sigma(\eta)+ \log(\eta)\\
\eta-\div(a(x)D\eta) -\div (\eta b(x)) - \div(\eta Dv) + \sigma(v+\Delta^{2k}v) - 1
\end{bmatrix}.
\end{equation}

For our analysis, we work with classical solutions $(\mv,\uv)\in C^\infty(\Tt^d;\Rr^+)\times C^\infty(\Tt^d)$ of \eqref{Rmfg}, whose existence follows from the argument in~\cite[Proposition~3.1]{FG2}.

\begin{remark}\label{rem:entropy}
Before establishing the a priori estimates for classical solutions $(m_\sigma, u_\sigma) \in C^\infty(\Tt^d; \mathbb{R}^+) \times C^\infty(\Tt^d)$ of the regularized system \eqref{Rmfg}, we recall a standard analytic inequality for the entropy density. 
As established in \cite{FG2}, the singular penalization term $\beta_\sigma$ is designed to ensure that $m_\sigma > 0$ holds uniformly on $\Tt^d$, rendering the term $\log(m_\sigma)$ well-defined.
Since $s \leq C_\delta + \delta\, s\log s$ for all $s > 0$ and any $\delta > 0$ (for some $C_\delta > 0$), and $\log s \leq s - 1$ for all $s > 0$, integrating over $\Tt^d$ (which has finite measure) yields: for any $\delta > 0$, there exists $C_\delta > 0$, independent of $\sigma$, such that
\begin{equation}\label{entropy}
\max\left\{ \int_{\Tt^d} m_\sigma \,\dx,\; \int_{\Tt^d} \log(m_\sigma) \,\dx \right\}
\leq C_\delta + \delta \int_{\Tt^d} m_\sigma \log(m_\sigma) \,\dx.
\end{equation}
\end{remark}
We now begin by establishing first-order estimates. 

\begin{lemma}[First-order Estimates]\label{Lemma:FOE}
    For $\sigma\in(0,1)$, let $(\mv,\uv)\in C^\infty(\Tt^d;\Rr^+)\times C^\infty(\Tt^d)$ be a classical solution to the regularized system~\eqref{Rmfg} on $\Tt^d$. Then, there exists a constant $C$, depending only on the problem data, such that
\begin{align}
   &(a)\enspace \frac{1}{2}\int_{\Tt^d}|D\uv|^2(\mv+1)\,\dx+\int_{\Tt^d}-\beta_\sigma(\mv)\,\dx+\int_{\Tt^d}\beta_\sigma(\mv)(\mv-\sigma)\,\dx\hfill\nonumber\\
&                  \quad\qquad +\, \sigma \int_{\Tt^d} \Big(|\Delta^k\uv|^2+u_\sigma^2+|\Delta^k\mv|^2+m_\sigma^2\Big)\,\dx\leq  C,\label{eq:1storder1}\\
&(b)\enspace\bigg|\int_{\Tt^d}\mv\log(\mv)\,\dx\bigg| \leq C,\label{eq:1storder2}\\
&(c)\enspace \int_{\Tt^d}\log(\mv)\,\dx \geq -C.\label{eq:1storder3}
\end{align}
  
\end{lemma}
\begin{proof}
For convenience, we drop the dependence on $\sigma$ and write $(m,u)$ instead of $(\mv,\uv)$. We start by multiplying the first equation of \eqref{Rmfg} by $(m-\sigma-1)$ and the second one by $u$. Summing the resulting equations and integrating over $\Tt^d$, we observe that some of the diffusion cross-terms cancel out. Furthermore, the transport term $\int_{\Tt^d} u \div(m Du)$ combines with the Hamiltonian term $-\tfrac12\int_{\Tt^d} (m-\sigma-1)|Du|^2$ after integrating by parts to yield a positive definite quadratic form. Precisely, we deduce that
\begin{align}\label{eq:first1}
 \int_{\Tt^d}& \Big(\frac12 m + \frac{\sigma+1}{2}\Big) |Du|^2 \,\dx
+ \sigma \int_{\Tt^d} \big(u^2 + |\Delta^ku|^2 + |\Delta^km|^2+m^2\big) \,\dx\nonumber \\
&+ \int_{\Tt^d} (m-\sigma)  \beta_\sigma(m)  \,\dx -    \int_{\Tt^d} \beta_\sigma(m)  \,\dx  + \int_{\Tt^d} m\log(m)  \,\dx \nonumber \\
&=-\sigma \int_{\Tt^d}  u \,\dx 
+\sigma(\sigma+1) \int_{\Tt^d}m\,\dx + \int_{\Tt^d} m V(x) \,\dx - (\sigma+1)\int_{\Tt^d}  V(x)\,\dx\nonumber\\
&\quad + (\sigma+1)\int_{\Tt^d} \log(m)  \,\dx  - (\sigma+1)\int_{\Tt^d} b(x) \cdot Du  \,\dx\nonumber\\
&\leq \frac\sigma2\int_{\Tt^d} u^2\,\dx +\frac12 + \big(2 + \Vert V\Vert_{L^\infty(\Tt^d)}\big) \int_{\Tt^d} m\,\dx + 2\Vert V\Vert_{L^\infty(\Tt^d)}\nonumber\\
&\quad + (\sigma+1)\int_{\Tt^d} \log(m)  \,\dx  + \frac{(\sigma+1)}{4}\int_{\Tt^d} |Du|^2 \,\dx +  \Vert b\Vert^2_{L^\infty(\Tt^d)},
\end{align}
where, in the last estimate, we used  Young's inequality together with the fact that $0<\sigma <1$. We further observe that we can use \eqref{entropy} to find a constant, $C>0$, that depends only on $\Vert V\Vert_{L^\infty(\Tt^d)}$, such that
\begin{align}
    & \big(2 + \Vert V\Vert_{L^\infty(\Tt^d)}\big) \int_{\Tt^d} m\,\dx \leq \frac{1}{4}\int_{\Tt^d} m \log m \,\dx + C,\label{eq:first2}\\
    & (\sigma+1)\int_{\Tt^d} \log(m)  \,\dx \leq \frac{1}{4}\int_{\Tt^d} m \log m \,\dx + C.\label{eq:first2a}
    \end{align}
Consequently, in view of \eqref{eq:first1} and \eqref{eq:first2} first, and \eqref{eq:first2a} second,  we have that
\begin{align}
    \int_{\Tt^d}& \Big(\frac12 m + \frac{\sigma+1}{4}\Big) |Du|^2 \,\dx
+ \frac\sigma2 \int_{\Tt^d} \big(u^2 + |\Delta^ku|^2 + |\Delta^km|^2+m^2\big) \,\dx\nonumber \\
&+ \int_{\Tt^d} (m-\sigma)  \beta_\sigma(m)  \,\dx -    \int_{\Tt^d} \beta_\sigma(m)  \,\dx  + \frac34\int_{\Tt^d} m\log(m)  \,\dx \nonumber \\
& \leq (\sigma+1)\int_{\Tt^d} \log(m)  \,\dx + C\label{eq:first3a}\\
& \leq \frac{1}{4}\int_{\Tt^d} m \log m \,\dx + C \label{eq:first3b}
\end{align}
for some constant $C$ that depends only on $\Vert V\Vert_{L^\infty(\Tt^d)}$ and $\Vert b\Vert_{L^\infty(\Tt^d;\Rr^d)}$.
Finally, observing that  $\inf_{s>0} s\log s > - C$ for some other constant $C>0$ independent of $s$,  and that all terms on the left-hand side of \eqref{eq:1storder1} are nonnegative, we obtain \eqref{eq:1storder1}  and \eqref{eq:1storder2} from \eqref{eq:first3b} (by absorbing the entropy term $\frac{1}{4} \int_{\Tt^d} m \log( m)$ into the left-hand side). In turn, these two estimates and \eqref{eq:first3a} yield \eqref{eq:1storder3}. 
\end{proof}

Next, we prove  second-order estimates. We note that to prove these estimates, we require the Lipschitz norm $\|Da\|_{L^\infty(\Tt^d)}$ to be less than $1$.
\begin{lemma}[Second-order Estimates]\label{Lemma:SOE}
Assume that $b\in C^\infty(\Tt^d;\Rr^d)$, $V\in C^\infty(\Tt^d)$, and $a\in C^\infty(\Tt^d)$ is such that $\min_{\Tt^d} a >0$ and  $\|Da\|_{L^\infty(\Tt^d)}<1$.
For $\sigma\in(0,1)$, let $(\mv,\uv)\in C^\infty(\Tt^d;\Rr^+)\times C^\infty(\Tt^d)$ be a classical solution to the regularized system~\eqref{Rmfg} on $\Tt^d$.  
Then, there exists a positive constant, $C$, depending only on the problem data and independent of $\sigma$, such that
\begin{align*}
    &\int_{\Tt^d} \mv|D^2 \uv|^2 \,\dx
+ \frac12\int_{\Tt^d} \frac{|D\mv|^2}{\mv} \,\dx
+\int_{\Tt^d}\beta'_\sigma(\mv)|D\mv|^2 \,\dx \\
&\quad+\,\sigma  \int_{\Tt^d} \big(|D\uv|^2+|\Delta^k D\uv|^2+|D\mv|^2+|\Delta^k D\mv|^2\big) \,\dx
\leq  C.
\end{align*}
\end{lemma}

\begin{proof}
For convenience, we drop the dependence on $\sigma$ and write $(m,u)$ instead of $(\mv,\uv)$.
    Differentiating the first equation in \eqref{Rmfg} twice with respect to $x_j$, multiplying by $m$, and integrating by parts, yields
\begin{align}
    & \int_{\Tt^d} u_{x_j} m_{x_j}\,\dx - 
    \int_{\Tt^d} \big( a_{x_jx_j} Du + 2 a_{x_j} D u_{x_j} +  a D u_{x_j x_j} \big)\cdot  Dm \,\dx \nonumber\\
    &\quad - \int_{\Tt^d} m \big(  b_{x_jx_j} \cdot Du +2 b_{x_j} \cdot Du_{x_j} + b\cdot D u_{x_jx_j}\big)\,\dx - \int_{\Tt^d} m D u_{x_j x_j} \cdot D u \,\dx\nonumber \\
    &\quad - \int_{\Tt^d} m|D u_{x_j}|^2 \,\dx + \int_{\Tt^d} V_{x_j} m_{x_j}  \,\dx - \sigma  \int_{\Tt^d} \big(|m_{x_j}|^2 +|\Delta^k m_{x_j}|^2\big)  \,\dx \nonumber\\
    & \quad - \int_{\Tt^d} \beta_\sigma'(m) m_{x_j}^2  \,\dx = \int_{\Tt^d} \frac{|m_{x_j}|^2}{m}  \,\dx.\label{E8}
\end{align}
On the other hand,  multiplying the second equation in \eqref{Rmfg} by $u_{x_jx_j}$ and integrating by parts, we conclude that 
\begin{align}
    & -\int_{\Tt^d}u_{x_j}m_{x_j}\,\dx + \int_{\Tt^d}a Du_{x_jx_j}\cdot Dm\,\dx + \int_{\Tt^d} m b \cdot Du_{x_jx_j}\,\dx + \int_{\Tt^d} m Du \cdot Du_{x_jx_j}\,\dx \nonumber\\
    &\quad -\sigma  \int_{\Tt^d} \big(|u_{x_j}|^2 +|\Delta^k u_{x_j}|^2\big)  \,\dx   = 0.\label{E8a}
\end{align}
Adding \eqref{E8}  and \eqref{E8a} first, and then summing over $j$, it follows that
\begin{align}
    & \int_{\Tt^d} \Delta a D u\cdot  Dm  \,\dx + 2 \int_{\Tt^d} (Da)^T D^2u Dm \,\dx + \int_{\Tt^d} m\Delta b  \cdot  Du  \,\dx  + 2\int_{\Tt^d} m Db :    D^2u  \,\dx \nonumber\\
    &\quad + \int_{\Tt^d} m|D^2 u|^2   \,\dx - \int_{\Tt^d} DV \cdot Dm  \,\dx + \sigma \int_{\Tt^d}\big(|Du|^2+|\Delta^kDu|^2+ |Dm|^2  + |\Delta^k Dm|^2 \big) \,\dx \nonumber\\
    &\quad + \int_{\Tt^d} \beta_\sigma'(m) |Dm|^2 \,\dx + \int_{\Tt^d} \frac{|Dm|^2}{m}  \,\dx = 0.\label{E9}
\end{align} 
Next, we apply Young's inequality to control the terms with  indefinite sign to absorb them into the positive terms, bounding the remainder by a constant independent of $\sigma$. We start by observing that \eqref{entropy} and \eqref{eq:1storder2} yield
\begin{align}\label{eq:estm}
    \int_{\Tt^d} m \,\dx \leq C.
\end{align}
Fix $0<\delta<1$. We have that
\begin{align*}
    \int_{\Tt^d} DV\cdot Dm  \,\dx &= \int_{\Tt^d} (\delta^{-\frac12} m^\frac12 DV)\cdot(\delta^{\frac12} m^{-\frac12} Dm)  \,\dx\\
    &\leq  \frac{\Vert DV\Vert^2_{L^\infty{(\Tt^d)}}}{2\delta} \int_{\Tt^d} m \,\dx + \frac{\delta}{2} \int_{\Tt^d}\frac{|Dm|^2}{m} \,\dx \leq \frac{C}{\delta} + \frac{\delta}{2} \int_{\Tt^d}\frac{|Dm|^2}{m} \,\dx.
\end{align*}

Similarly, recalling \eqref{eq:1storder1},
\begin{align*}
    -\int_{\Tt^d} \Delta a D u\cdot  Dm  \,\dx & = -\int_{\Tt^d} \Delta a(\delta^{-\frac12} m^\frac12 Du)\cdot(\delta^{\frac12} m^{-\frac12} Dm)  \,\dx\\
    &\leq  \frac{\Vert D^2a\Vert^2_{L^\infty{(\Tt^d)}}}{2\delta} \int_{\Tt^d} m |Du|^2 \,\dx + \frac{\delta}{2} \int_{\Tt^d}\frac{|Dm|^2}{m} \,\dx\\
    &\leq \frac{C}{\delta} + \frac{\delta}{2} \int_{\Tt^d}\frac{|Dm|^2}{m} \,\dx,
\end{align*}

\begin{align*}
    -2\int_{\Tt^d} m Db :    D^2u  \,\dx & = -\int_{\Tt^d} (2\delta^{-\frac12} m^\frac12 Db):(\delta^{\frac12} m^{\frac12} D^2u)  \,\dx\\
    &\leq  \frac{2\Vert Db\Vert^2_{L^\infty{(\Tt^d)}}}{\delta} \int_{\Tt^d} m \,\dx + \frac{\delta}{2} \int_{\Tt^d}m |D^2u|^2 \,\dx\\
    &\leq \frac{C}{\delta} + \frac{\delta}{2} \int_{\Tt^d}m |D^2u|^2 \,\dx
\end{align*}
and
\begin{align*}
    -\int_{\Tt^d} m \Delta b \cdot    Du  \,\dx & = -\int_{\Tt^d} ( m^\frac12 \Delta b) \cdot( m^{\frac12} Du)  \,\dx\\
    &\leq  \frac{\Vert \Delta b\Vert^2_{L^\infty{(\Tt^d)}}}{2} \int_{\Tt^d} m \,\dx + \frac{1}{2} \int_{\Tt^d}m |Du|^2 \,\dx \leq C.
\end{align*}

Finally, 
\begin{align*}
    \bigg| 2 \int_{\Tt^d} (Da)^T D^2u Dm \,\dx \bigg| & \leq  2\int_{\Tt^d} \big( \Vert Da\Vert^\frac12_{L^\infty(\Tt^d)} m^\frac12 |D^2u| \big ) \big( \Vert Da\Vert^\frac12_{L^\infty(\Tt^d)} m^{-\frac12} |Dm| \big)  \,\dx\\
    &\leq  \Vert Da\Vert_{L^\infty(\Tt^d)} \int_{\Tt^d} m |D^2 u|^2 \,\dx + \Vert Da\Vert_{L^\infty(\Tt^d)} \int_{\Tt^d}\frac{|Dm|^2}{m} \,\dx.
\end{align*}

Let $\tau=1- \Vert Da\Vert_{L^\infty(\Tt^d)}$. By our assumption on $Da$, we have $\tau>0$. Then, taking $\delta= \frac\tau2$ in the estimates above, we conclude from \eqref{E9} that
\begin{align*}
   &\frac{3\tau}{4}\int_{\Tt^d}m |D^2u|^2 \,\dx + \sigma \int_{\Tt^d}\big(|Du|^2+|\Delta^kDu|^2+ |Dm|^2  + |\Delta^k Dm|^2 \big) \,\dx \\
   &\quad+ \int_{\Tt^d} \beta_\sigma'(m) |Dm|^2 \,\dx + \frac\tau2\int_{\Tt^d} \frac{|Dm|^2}{m}  \,\dx \leq C + \frac{C}{\tau}, 
\end{align*}
from which the conclusion follows.\end{proof}
This is precisely the estimate that replaces Assumption~\ref{2.12} in the quadratic regime.
In the following corollary, we collect the uniform bounds on $(\mv,\uv)$ required for the existence result in Theorem~\ref{Th:IR}.

\begin{cor}\label{cor1} 
Assume that $b\in C^\infty(\Tt^d;\Rr^d)$, $V\in C^\infty(\Tt^d)$, and $a\in C^\infty(\Tt^d)$ is such that $\min_{\Tt^d} a >0$ and  $\|Da\|_{L^\infty(\Tt^d)}<1$. 
For $\sigma\in(0,1)$, let $(\mv,\uv)\in C^\infty(\Tt^d;\Rr^+)\times C^\infty(\Tt^d)$ be a classical solution to the regularized system~\eqref{Rmfg} on $\Tt^d$.  
Then, there exists a positive constant, $C$, depending only on the problem data and independent of $\sigma$, such that 
\begin{equation}\label{cor1ES.}
    \|(\mv)^\frac{1}{2}\|_{W^{1,2}(\Tt^d)}+\|\uv\|_{W^{1,2}(\Tt^d)}\leq  C.
\end{equation}
\end{cor}
\begin{proof}
The bounds in Lemma~\ref{Lemma:SOE} yield
\[
\|D((\mv)^\frac{1}{2})\|_{L^2(\Tt^d)}^2=\frac14 \int_{\Tt^d}\frac{|D\mv|^2}{\mv} \,\dx \leq  C.
\]
On the other hand, as justified in \eqref{eq:estm}, we have $\int_{\Tt^d}\mv \,\dx \leq  C$; hence,
$\|(\mv)^\frac{1}{2}\|_{L^2(\Tt^d)}^2=\int_{\Tt^d}\mv \,\dx \leq  C$. Therefore,
$$\|(\mv)^\frac{1}{2}\|_{W^{1,2}(\Tt^d)}\leq  C.$$

To prove the uniform bound for $u$,  we observe that Lemma~\ref{Lemma:FOE} yields $\|D\uv\|_{L^2(\Tt^d)}\leq  C$. Moreover,
integrating the first equation in \eqref{Rmfg} over $\Tt^d$ and using periodicity together with \eqref{entropy}
 and Lemma~\ref{Lemma:FOE}, we conclude that
\[
\Big|\int_{\Tt^d}\uv \,\dx\Big|\leq  C.
\]
By the Poincaré--Wirtinger inequality on the torus,
\[
\|\uv\|_{L^2(\Tt^d)}\leq  C\big(\|D\uv\|_{L^2(\Tt^d)}+|\bar u_\sigma|\big)\leq  C,
\]
where \(\bar u_\sigma = \frac{1}{|\Tt^d|} \int_{\Tt^d} \uv \,\dx\). Hence, $\|\uv\|_{W^{1,2}(\Tt^d)}\leq  C$.
\end{proof}

We now present our main existence result for the stationary quadratic regime. We establish the existence of weak solutions to Problem~\ref{Qproblem-log} in the sense of Definition~\ref{def:weaksol} with integrability exponents $r = \frac{2^*}{2}$ and $\gamma = 2$. This result represents a significant improvement over the general existence framework established in \cite{FG2}, which typically guarantees weak solutions only in  $L^1(\Tt^d) \times W^{1,2}(\Tt^d)$. By obtaining higher integrability for the density $\tilde{m}$, we provide the necessary regularity to apply the linearization argument required for uniqueness.
\begin{theorem}\label{Th:IR}
   Under the assumptions of Problem~\ref{Qproblem-log}, assume further that    $\|Da\|_{L^\infty(\Tt^d)}<1$.
   Then, there exists a weak solution $(\tilde m,\tilde u)$ to Problem~\ref{Qproblem-log} in the sense of Definition~\ref{def:weaksol} such that
\[
\tilde u\in W^{1,2}(\Tt^d),\qquad \tilde m^{1/2}\in W^{1,2}(\Tt^d).
\]
Moreover,
\[
\tilde m\in L^{\frac{2^*}{2}}(\Tt^d)\quad\text{if } d>2,
\]
while
\[
\tilde m\in L^q(\Tt^d)\quad\text{for every finite } q \quad\text{if } d\le2.
\]
\end{theorem}

\begin{proof}
The proof proceeds in two steps as follows: in Step 1, we extract a convergent subsequence from the regularized approximations using the uniform estimates established in Corollary~\ref{cor1};  subsequently, in Step 2, we identify the limit as a weak solution of Problem~\ref{Qproblem-log} via the Minty monotonicity device. 

Let $(\mv,\uv)\in C^\infty(\Tt^d)\times C^\infty(\Tt^d)$ be a classical solution to the regularized system~\eqref{Rmfg}. 

\medskip
\noindent\textbf{Step 1:} We establish the compactness of the regularized sequence $(\mv, \uv)$, and extract a limit pair $(\tim,\tiu) \in L^\frac{2^*}{2}(\Tt^d)\times W^{1,2}(\Tt^d)$ satisfying the mass constraint $\int_{\Tt^d} \tim\,\dx=1$.

By Corollary~\ref{cor1}, \((\mv)^{1/2}\) is bounded in \(W^{1,2}(\Tt^d)\). Hence, up to a subsequence (not relabeled),
\[
(\mv)^\frac{1}{2}\rightharpoonup \omega \quad \text{weakly in } W^{1,2}(\Tt^d)
\]
for some $\omega\in W^{1,2}(\Tt^d)$. Using the Sobolev embedding theorem, we further have that
\[
(\mv)^\frac{1}{2}\rightarrow \omega \quad \text{strongly in } L^q(\Tt^d)
\]
for any $q < 2^*$ if $d > 2$, or $q < \infty$ if $d \leq 2$. In particular, by the compact embedding from $W^{1,2}(\Tt^d)$ to $L^2(\Tt^d)$, we obtain strong convergence in $L^2(\Tt^d)$, which implies that $\mv \to \omega^2$ strongly in $L^1(\Tt^d)$. Consequently, defining $\tim = \omega^2$ yields  $\tim^{1/2} = \omega \in W^{1,2}(\Tt^d)$ and $\tim \in L^{\frac{2^*}{2}}(\Tt^d)$.

Similarly, we conclude that there exists $\tiu\in W^{1,2}(\Tt^d)$ such that $\uv \rightharpoonup \tiu$ weakly in $W^{1,2}(\Tt^d)$. Furthermore, once we integrate the second equation in \eqref{Rmfg} and take $\sigma \rightarrow0^+$, we obtain that $\int_{\Tt^d} \tim\,\dx=1$.
    
\medskip
\noindent\textbf{Step 2:} We prove that $(\tim,\tiu)$ is a weak solution of Problem~\ref{Qproblem-log}. To that end, we use the Minty--Browder method. Let $(\eta,v)\in C^\infty(\Tt^{d};\Rr^+)\times C^\infty(\Tt^{d})$. Then, because $A_\sigma \begin{bmatrix}\mv\\ \uv\end{bmatrix}=0$ and $A_\sigma$ is monotone, we have 
   \begin{align*}     
   0&\leq  \Big\langle 
    \begin{bmatrix}\eta\\ v\end{bmatrix}
    -
    \begin{bmatrix}\mv\\ \uv\end{bmatrix},
    A_\sigma\begin{bmatrix}\eta\\ v\end{bmatrix}
    \Big\rangle\\
    &=
    \int_{\Tt^{d}}(\eta-\mv) \big(-v+\div(a(x)Dv)-b(x)\cdot Dv -\frac12|Dv|^{2}-V(x)+\log\eta\big) \,\dx\\
    &\quad+\int_{\Tt^{d}}(\eta-\mv)\big( \sigma (\eta+\Delta^{2k}\eta) + \beta_\sigma(\eta)\big)+(v-\uv)  \big(\sigma (v+\Delta^{2k}v)\big) \,\dx\\
    &\quad+\int_{\Tt^{d}}(v-\uv)  \big(\eta-1-\div(a(x)D\eta)-\div (\eta b(x))-\div(\eta Dv)\big) \,\dx.
   \end{align*}
   Letting $\sigma\rightarrow0^+$,
   the convergences established above together with the convergence
    $\sigma (\eta+\Delta^{2k}\eta) + \beta_\sigma(\eta)+ \sigma (v+\Delta^{2k}v)\rightarrow 0$ uniformly on $\Tt^d$, which holds due to the regularity of $(\eta,v)$ and because  $\min_{\Tt^d} \eta>0$, yield
      \begin{align*}     
   &\int_{\Tt^{d}}(\eta-\tim) \big(-v+\div(a(x)Dv) -b(x)\cdot Dv -\frac12|Dv|^{2}-V(x)+\log\eta\big) \,\dx\\
    &\quad+\int_{\Tt^{d}}(v-\tiu)  \big(\eta-1-\div(a(x)D\eta)-\div (\eta b(x))-\div(\eta Dv)\big) \,\dx\ge 0.
   \end{align*}
    Therefore, $(\tim,\tiu)\in L^{\frac{2^*}{2}}(\Tt^d)\times W^{1,2}(\Tt^d)$ is a weak solution of Problem~\ref{Qproblem-log} in the sense of Definition~\ref{def:weaksol}.   
\end{proof}

\subsection{Weak-Strong Uniqueness}\label{sec6.2}
In this section, we adapt the monotonicity machinery to Problem~\ref{Qproblem} for which we prove the following uniqueness result. 

\begin{theorem}\label{Th:uniq-Q}
    Assume that   $(\tilde m,\tilde u)\in L^{2}(\Tt^{d})\times W^{1,2}(\Tt^{d})$ is a  weak solution to Problem~\ref{Qproblem} (in the sense of Definition~\ref{def:weaksol})  and that \( (m, u) \in W^{1,\infty}(\Tt^d) \times W^{2,\infty}(\Tt^d) \) is a strong solution to Problem~\ref{Qproblem} (in the sense of Definition~\ref{def:strongsol-S}). Assume further that there exists a constant $c_0>0$ such that \(
  m(x)\ge c_0\) for all  $x\in\Tt^d$, and that $D_m F(x,s)>0$ for all $x\in\Tt^d$ and all $s\in\mathbb{E}$.
Then, \[m=\tim \quad \text{and}\quad u=\tiu.\]
\end{theorem}

\begin{remark}
The uniqueness result in Theorem~\ref{Th:uniq-Q} requires weak solutions to Problem~\ref{Qproblem} to possess higher regularity than those typically prescribed by Definition~\ref{def:weaksol}. To date, the existence of solutions with this improved regularity for $\tilde{m}$ has been established in two primary settings:
\begin{itemize}
    \item \textbf{Power-type couplings:} In~\cite[Theorem~6.2]{FG2}, for couplings of the form $F(x,m) \sim m^{\alpha}$ ($\alpha > 0$) with a constant diffusion coefficient $a(x) \equiv c > 0$. In this case, $\tilde{m} \in L^{\frac{(\alpha+1)2^*}{2}}(\Tt^d)$ (as usual when $d\leq 2$, $2^*$ can be replaced by any arbitrary large exponent $q$). The required regularity $\tilde{m} \in L^2(\Tt^d)$ is satisfied whenever $\frac{(\alpha+1)2^*}{2}\geq 2$ holds. This occurs if $d \leq \frac{4}{1-\alpha}$ for $0 < \alpha < 1$, and for $d \geq 1$ for $\alpha \geq 1$.

    \item \textbf{Logarithmic couplings:} In Theorem~\ref{Th:IR} of the present work, for $F(x,m) = -V(x) + \log m$ and a non-constant diffusion coefficient $a(\cdot)$. This regularity is guaranteed provided $\|Da\|_{L^\infty(\Tt^d)} < 1$ and $d \leq 4$.
\end{itemize}
 We further observe that while the higher integrability of the density $\tilde{m}$ is a specific focus of the preceding subsection, the $W^{1,2}(\Tt^d)$ regularity for the value function $\tilde{u}$ in the context of quadratic Hamiltonians was already established in a general monotone setting in \cite{FG2}. 
\end{remark}

\begin{proof}[Proof of Theorem~\ref{Th:uniq-Q}]
The proof follows the same strategy as that of Theorem~\ref{Th:uniq-stat} in Section~\ref{sect:proofuniqstat}; therefore, we present only the key steps and expressions.

Using Definition~\ref{def:weaksol} and the regularity of the weak solution, $(\tilde m,\tilde u)\in L^{2}(\Tt^{d})\times W^{1,2}(\Tt^{d})$, a mollification argument allows us to conclude that, as in \eqref{E4}, 
\begin{align}\label{EQ:WeakQ}
    &\int_{\Tt^d} (\eta - \tilde{m})\Big(-v + \div(a(x) Dv) -b(x)\cdot Dv - \tfrac{1}{2}|Dv|^2 + F(x,\eta)\Big) \,\dx \nonumber \\
    &\quad + \int_{\Tt^d} \Big[(v - \tilde{u})(\eta-1) + (Dv - D\tilde{u}) \cdot \big(a(x)D\eta + \eta Dv + \eta b(x)\big)\Big] \,\dx \ge 0
\end{align}
holds for  all $(\eta,v)\in W^{1,\infty}(\Tt^d) \times W^{2,\infty}(\Tt^d)$ with $\essinf_{\Tt^d} \eta>0$.
On the other hand, Definition~\ref{def:strongsol-S} and the regularity imposed on the strong solution,  $(m,u) \in W^{1,\infty}(\Tt^d) \times W^{2,\infty}(\Tt^d)$, implies that, as in \ref{E5}, 
\begin{align}\label{EQ:StrongQ}
    &\int_{\Tt^d} (\eta - \tilde{m})\Big(-u + \div(a(x) Du) -b(x)\cdot Du - \tfrac{1}{2}|Du|^2 + F(x,m)\Big) \,\dx \nonumber \\
    &\quad + \int_{\Tt^d} \Big[(v - \tilde{u})(m-1) + (Dv - D\tilde{u}) \cdot \big(a(x)Dm + m Du + m b(x)\big)\Big] \,\dx = 0
\end{align}
holds for all  $(\eta,v)\in W^{1,\infty}(\Tt^d) \times W^{2,\infty}(\Tt^d)$ with $\essinf_{\Tt^d} \eta>0$.

Thus, subtracting \eqref{EQ:StrongQ} from \eqref{EQ:WeakQ}, we conclude for all $(\eta,v)\in W^{1,\infty}(\Tt^d) \times W^{2,\infty}(\Tt^d)$ with $\essinf_{\Tt^d} \eta>0$ that 
\begin{align}
    \mathcal{I}[\eta,v]\geq 0,\label{eq:stron-weak-var}
\end{align}
where
\begin{align*}
    \mathcal{I}[\eta,v]  = &\int_{\Tt^d} (\eta - \tilde{m})\big(u-v - \div(a(x) (Du-Dv)) +b(x)\cdot (Du - Dv)\\&\hspace{30mm} + \tfrac{1}{2}|Du|^2 -\tfrac{1}{2}|Dv|^2 - F(x,m) + F(x,\eta)\big) \,\dx\\
    &+ \int_{\Tt^d} \big[(v - \tilde{u})(\eta-m) + (Dv - D\tilde{u}) \cdot \big(a(x)(D\eta - Dm) + \eta Dv - m Du \\&\hspace{30mm}  + (\eta - m) b(x)\big)\big] \,\dx.
\end{align*}

Then, fixing $(\bar 
\eta, \bar v) \in C^2(\Tt^d) \times C^2(\Tt^d)$, the uniform lower bound on $m$ allows us to take 
\[ \eta = m + \varepsilon \bar \eta \quad \text{and} \quad v = u + \varepsilon \bar v\]
in \eqref{eq:stron-weak-var} for any sufficiently small $|\varepsilon|$. Thus, 
as in the proof of Theorem~\ref{Th:uniq-stat}, we must have $\Ii'(0)=0$, where $\Ii(\varepsilon) = \mathcal{I}[m + \varepsilon \bar \eta, u + \varepsilon \bar v]$. Consequently,
\begin{align}
    &\int_{\Tt^d} (m- \tilde m) \big( -\bar v + \div (a(x) D\bar v) - b(x) \cdot D\bar v -Du \cdot D\bar v + \bar \eta D_m F(x,m)\big)\,\dx\nonumber\\
    &\quad + \int_{\Tt^d} \big[(u- \tilde u) \bar \eta  + (Du - D\tilde u) \cdot \big( a(x) D\bar \eta + m D\bar v + \bar \eta Du + \bar\eta b(x)\big)\big] \,\dx = 0.\label{eq:estinepsi1}
\end{align}
To handle the higher-order diffusion terms, we proceed as Theorem~\ref{Th:uniq-stat} and take in the preceding estimate the $C^2$ approximations $\eta_{  \delta}$ and $v_{  \delta}$ of $(m - \tilde m)$ and $(u - \tilde u)$, respectively,  as in \eqref{eq:molletaepsi} and \eqref{eq:mollvepsi} to obtain 
\begin{align}
    &\int_{\Tt^d} \Big[ (m- \tilde m) \big( \div (a(x) Dv_{  \delta}) - b(x) \cdot Dv_{  \delta} \big)+ (Du - D\tilde u) \cdot \big( a(x) D\eta_{  \delta}  + \eta_{  \delta} b(x)\big)\Big] \,\dx \nonumber
    \\&\quad = \int_{\Tt^d} \Big[ (m- \tilde m) \big( Da(x) \cdot Dv_{  \delta} - b(x) \cdot Dv_{  \delta} \big)+ (Du - D\tilde u) \cdot \big(  \eta_{  \delta} b(x) - Da(x) \eta_{  \delta}\big)\Big] \,\dx,\label{eq:estinepsi2}
\end{align}
where we used the identities $\div(a(x) Dv_{  \delta}) = a(x) \Delta v_{  \delta} + Da(x) \cdot Dv_{  \delta}$ and $a(x) D\eta_{  \delta} = D(a(x) \eta_{  \delta}) - Da(x)\eta_{  \delta}$.

We conclude from \eqref{eq:estinepsi1} and \eqref{eq:estinepsi2} that
\begin{align}
    &\int_{\Tt^d} (m- \tilde m) \Big( -v_{  \delta}  -Du \cdot Dv_{  \delta} + \eta_{  \delta} D_m F(x,m)\Big)\,\dx\nonumber\\
    &\quad + \int_{\Tt^d} \Big[(u- \tilde u) \eta_{  \delta}  + (Du - D\tilde u) \cdot \big(  m Dv_{  \delta} + \eta_{  \delta} Du \big)\Big] \,\dx \nonumber\\
    &\quad + \int_{\Tt^d} \Big[ (m- \tilde m) \big( Da(x) \cdot Dv_{  \delta} - b(x) \cdot Dv_{  \delta} \big)+ (Du - D\tilde u) \cdot \big(  \eta_{  \delta} b(x) - Da(x) \eta_{  \delta}\big)\Big] \,\dx = 0.\label{eq:estinepsi3}
\end{align}
Finally, letting ${  \delta}\to0$ in the preceding estimate, and using  the convergence properties of the mollifiers (Section \ref{R8}), yields
\begin{align*}
    \int_{\Tt^d} \Big[ (m - \tilde{m})^2 D_m F(x, m) + m |Du - D\tilde{u}|^2 \Big] \,\dx = 0.
\end{align*}
Because $m \ge c_0 > 0$ and $D_m F(x,m(x)) > 0$ by assumption, we conclude that $Du = D\tilde{u}$ and $m = \tilde{m}$ a.e.. This implies that $u - \tilde{u} = C$ for some constant $C$. Substituting these results back into  \eqref{eq:estinepsi1} forces $C=0$; hence,  $u=\tilde{u}$ a.e., which concludes the proof.
\end{proof}

\subsection{Strong convergence of monotone mean-field game systems}\label{sect:quad-conv}

In this section, we establish a stability result for the regularized system \eqref{Rmfg} introduced in Section \ref{sec:quad-reg}. While the a priori estimates derived previously yield the existence of weak solutions via compactness, such arguments typically guarantee only weak convergence (up to subsequences). 
We employ a regularity trade-off argument: we show that if Problem~\ref{Qproblem-log} admits a sufficiently regular solution $(m,u)$, then the whole family of regularized solutions $(m_\sigma,u_\sigma)$ to \eqref{Rmfg} converges strongly to $(m,u)$ as $\sigma\to0$. 
We take $(m,u)\in W^{2k,2}(\Tt^d)\times W^{2k,2}(\Tt^d)$ solving Problem \ref{Qproblem-log} in $\mathcal D'(\Tt^d)$.
This approach automatically yields the weak-strong uniqueness for the weak solutions produced by the particular approximation scheme. Thus, in this logarithmic quadratic setting, the abstract weak--strong uniqueness principle is complemented by a concrete strong convergence result for the monotone regularization.

\begin{theorem}\label{Th:strong-conv}
Let $k\in\Nn$  be such that $2k-4 >\frac d2+1 $, and assume that there exists a (strong) solution $(m, u) \in W^{2k,2}(\Tt^d) \times W^{2k,2}(\Tt^d)$ to Problem~\ref{Qproblem-log}  with $\essinf_{\Tt^d}m  > 0$. For $\sigma\in(0,1)$, let
 $(m_\sigma, u_\sigma)\in C^\infty(\Tt^d)\times C^\infty(\Tt^d)$ be the  classical solution to the regularized system \eqref{Rmfg}.
    Then, as $\sigma \to 0$,
    \begin{align}
        &\mv \to m \quad\text{ strongly in } L^1(\Tt^d),\label{eq:convergencem}\\
        &\uv \to u \quad\text{ strongly in }W^{1,2}(\Tt^d).\label{eq:convergenceu}
    \end{align}
 \end{theorem}

\begin{proof}
We start by observing that the condition on $k$ yields that $(m,u)$ satisfies \eqref{QMFG-log} pointwise in $\Tt^d$ by the  Sobolev embedding theorem. 

Let $c_0\in\Rr^+$ be such that $\min_{\Tt^d}m \ge c_0 > 0$, and fix $\sigma \in (0, c_0/2)$. For any such $\sigma$, we have $\beta_\sigma(m) = 0$. Moreover,  $(m, u)$ satisfies in $\mathcal{D}'(\Tt^d)$ the following version of \eqref{Rmfg} with residual terms:
\begin{equation}\label{eq:residual}
\begin{cases}
-u + \div(a(x)Du) -b(x)\cdot Du - \frac{1}{2}|Du|^2 \\
\quad - V(x) + \sigma(m + \Delta^{2k}m)
= -\log(m) + \sigma(m + \Delta^{2k}m),\\[1ex]
m - \div(a(x)Dm) -\div (mb(x)) - \div(m Du) + \sigma(u + \Delta^{2k}u)
= 1 + \sigma(u + \Delta^{2k}u).
\end{cases}
\end{equation}

We note that the regularity of $(m,u)$ allows us to test the equations in \eqref{eq:residual} against functions $\varphi \in W^{2k,2}(\Tt^d)$ for which the terms involving $\Delta^{2k}$ are, as usual, understood via repeated integration by parts:
\begin{align*}
    \langle \Delta^{2k}m, \varphi \rangle = \int_{\Tt^d} \Delta^{k}m\Delta^{k}\varphi\,\dx \quad \text{and} \quad \langle \Delta^{2k}u, \varphi \rangle = \int_{\Tt^d} \Delta^{k}u\Delta^{k}\varphi\,\dx.
\end{align*}
Moreover, for convenience, we use in the sequel the equivalent Hilbert norm on $W^{2k,2}(\Tt^d)$ given by
\[
\|w\|_{W^{2k,2}(\Tt^d)}^2 := \|w\|_{L^2(\Tt^d)}^2 + \|\Delta^{k}w\|_{L^2(\Tt^d)}^2,
\]
which matches the coercive terms produced by the regularization after integration by parts.

Next, we test the difference between the first equation in \eqref{eq:residual} and \eqref{Rmfg} against the test function $(m - m_\sigma)$ and the difference between the second equation in  \eqref{eq:residual} and \eqref{Rmfg} against the  test function $(u - u_\sigma)$. Integrating by parts and summing the resulting equations yields the following identity:
\begin{align}
    &- \frac{1}{2} \int_{\Tt^d} (m - m_\sigma)(|Du|^2 - |Du_\sigma|^2) \,\dx 
    + \int_{\Tt^d} (m - m_\sigma)(\beta_\sigma(m) - \beta_\sigma(m_\sigma))\,\dx \nonumber\\
    &\quad+\, \sigma\Big( \|u - u_\sigma\|^2_{W^{2k,2}(\Tt^d)} + \|m - m_\sigma\|^2_{W^{2k,2}(\Tt^d)}\Big) 
    + \int_{\Tt^d} (m - m_\sigma)(\log(m) - \log(m_\sigma)) \,\dx \nonumber \\
    &\quad+ \int_{\Tt^d} (Du - Du_\sigma)\cdot (mDu - m_\sigma Du_\sigma)\,\dx \nonumber\\
    &\quad= \sigma\int_{\Tt^d} \big[(u - u_\sigma)u + (\Delta^k u - \Delta^k u_\sigma)\Delta^{k}u\big]\,\dx \nonumber \\&\qquad+\,\sigma\int_{\Tt^d}
    \big[(m - m_\sigma)m + (\Delta^k m - \Delta^k m_\sigma)\Delta^{k}m\big]\,\dx , \label{TestingE}
\end{align}
where we note that the terms involving $a(\cdot)$ and $b(\cdot)$ canceled due to the adjoint structure of both systems \eqref{eq:residual} and \eqref{Rmfg}. 

Using the identity
\begin{equation*}
\begin{aligned}
   &- \frac{1}{2}  (m - m_\sigma)(|Du|^2 - |Du_\sigma|^2) + (Du - Du_\sigma)\cdot (mDu - m_\sigma Du_\sigma)\\
   &\quad= \frac12 m|Du|^2 + \frac12 m_\sigma|Du_\sigma|^2 + \frac12 m|Du_\sigma|^2 + \frac12 m_\sigma|Du|^2 - (m + m_\sigma) Du\cdot Du_\sigma\\
   &\quad= \frac12 (m + m_\sigma) |Du - Du_\sigma|^2
\end{aligned}
\end{equation*}
on the left-hand side of \eqref{TestingE} and Young's inequality to estimate the right-hand side of \eqref{TestingE} yields:
\begin{align*}
    &\frac{1}{2}\int_{\Tt^d} (m + m_\sigma) |Du - Du_\sigma|^2 \,\dx 
    +  \int_{\Tt^d}(m - m_\sigma) (\beta_\sigma(m) - \beta_\sigma(m_\sigma))\,\dx  \nonumber\\
    &\quad+ \int_{\Tt^d}(m - m_\sigma) (\log m - \log m_\sigma) \,\dx + \sigma\Big( \|u - u_\sigma\|^2_{W^{2k,2}(\Tt^d)} + \|m - m_\sigma\|^2_{W^{2k,2}(\Tt^d)}\Big) \nonumber\\
    &\quad \leq \sigma\int_{\Tt^d} \bigg( \frac{(u - u_\sigma)^2}{2} + \frac{u^2}{2} + \frac{(\Delta^k u - \Delta^k u_\sigma)^2}{2} + \frac{(\Delta^k u)^2}{2}\bigg)\,\dx \nonumber \\
    &\qquad + \sigma\int_{\Tt^d} \bigg( \frac{(m - m_\sigma)^2}{2} + \frac{m^2}{2} + \frac{(\Delta^k m - \Delta^k m_\sigma)^2}{2} + \frac{(\Delta^k m)^2}{2}\bigg)\,\dx\nonumber \\
    &\quad= \frac\sigma2 \Big( \|u - u_\sigma\|^2_{W^{2k,2}(\Tt^d)} + \|m - m_\sigma\|^2_{W^{2k,2}(\Tt^d)}\Big) +\frac\sigma2 \Big( \|u\|^2_{W^{2k,2}(\Tt^d)} + \|m\|^2_{W^{2k,2}(\Tt^d)}\Big). 
\end{align*}
Absorbing the first term on the right-hand side of the preceding identity into its  left-hand side, we deduce the following key estimate:
\begin{align}\label{eq:key-identity}
   & \frac{1}{2}\int_{\Tt^d} (m + m_\sigma) |Du - Du_\sigma|^2 \,\dx 
    +  \int_{\Tt^d}(m - m_\sigma) (\beta_\sigma(m) - \beta_\sigma(m_\sigma))\,\dx  \nonumber\\
    &\quad+ \int_{\Tt^d}(m - m_\sigma) (\log m - \log m_\sigma) \,\dx + \frac\sigma2\Big( \|u - u_\sigma\|^2_{W^{2k,2}(\Tt^d)} + \|m - m_\sigma\|^2_{W^{2k,2}(\Tt^d)}\Big)  \leq  C \sigma,
\end{align}
where $C= \tfrac12\big(\|u\|^2_{W^{2k,2}(\Tt^d)} + \|m\|^2_{W^{2k,2}(\Tt^d)}\big)$.

To prove the $L^1$ convergence of $m_\sigma$, we utilize the elementary inequality $(a - b)(\log a - \log b) \ge 4 (\sqrt{a} - \sqrt{b})^2$, for $a, b > 0$, that combined with the fact that all terms on the left-hand side of \eqref{eq:key-identity} are nonnegative, gives 
\begin{align*}
    4 \|\sqrt{m} - \sqrt{m_\sigma}\|_{L^2(\Tt^d)}^2 \leq \int_{\Tt^d} (m - m_\sigma)(\log m - \log m_\sigma) \,\dx \leq C\sigma.
\end{align*}
Thus, $m_\sigma^{1/2} \to m^{1/2}$ in $L^2(\Tt^d)$, from which we deduce \eqref{eq:convergencem}.

For the convergence of $u_\sigma$, we use the lower bound $m \ge c_0 > 0$ in the quadratic gradient term in \eqref{eq:key-identity} to get:
\begin{align}
    \frac{c_0}{2} \|Du - Du_\sigma\|_{L^2(\Tt^d)}^2 \leq \int_{\Tt^d} \frac{m + m_\sigma}{2} |D(u - u_\sigma)|^2 \,\dx \leq C\sigma.\label{eq:estDuepsi}
\end{align}    
To control the full $W^{1,2}$ norm, we must control the mean value. 
Integrating the Hamilton--Jacobi equations in \eqref{QMFG-log} and \eqref{Rmfg} over $\Tt^d$ and subtracting the former from the latter, we obtain
\begin{align}
    \int_{\Tt^d}(u - u_\sigma)\,\dx
    = & -\int_{\Tt^d} b(x)\cdot (Du - Du_\sigma)\,\dx-\frac{1}{2} \int_{\Tt^d}\big(|Du|^2 - |Du_\sigma|^2\big)\,\dx
      - \int_{\Tt^d}\beta_\sigma(m_\sigma)\,\dx\nonumber\\
      &- \int_{\Tt^d}\big(\log m_\sigma - \log m\big)\,\dx
      - \sigma\int_{\Tt^d} m_\sigma\,\dx,\label{eq:estmean}
\end{align}
where we used periodicity to eliminate the divergence and $\Delta^{2k}$ terms.
The first, second, and last terms on the right-hand side of \eqref{eq:estmean} tend to zero as $\sigma \to 0$ by \eqref{eq:estDuepsi} and \eqref{eq:convergencem}.
Next, we prove that also the third and fourth terms converge to zero as $\sigma \to 0$.

We have $\beta_\sigma(m)=0$ on $\Tt^d$ because $m\ge c_0$ on $\Tt^d$,  $\beta_\sigma(m_\sigma) = 0$ on the set $\{m_\sigma\ge \sigma\}$ by definition of $\beta_\sigma(\cdot)$, and $m-m_\sigma \geq \tfrac{c_0}{2}$ on the set $\{m_\sigma\le \sigma\}$ because $0<\sigma \leq \tfrac{c_0}{2}$. Hence, using the monotonicity of $\beta_\sigma(\cdot)$ and \eqref{eq:key-identity}, we have that
\begin{align*}
    0&\le -\int_{\Tt^d}\beta_\sigma(m_\sigma)\,\dx = \int_{\{m_\sigma\le \sigma\}}\big(\beta_\sigma(m)-\beta_\sigma(m_\sigma)\big)\,\dx\\
&\le \frac{2}{c_0}\int_{\{m_\sigma\le \sigma\}}(m-m_\sigma)\big(\beta_\sigma(m)-\beta_\sigma(m_\sigma)\big)\,\dx
\le C\sigma, 
\end{align*}
using \eqref{eq:key-identity}.
Therefore, $\big|-\int_{\Tt^d}\beta_\sigma(m_\sigma)\,\dx\big|\to0$ as $\sigma\to0$.

Similarly, using the bound $|\log a - \log b| \leq |a-b| (\min\{a,b\})^{-1}$, for $a, b>0$, together with the monotonicity of $\log(\cdot)$, we deduce that
\begin{align*}
    \bigg| - \int_{\Tt^d}(\log m_\sigma - \log m)\,\dx\bigg| &\leq 
      \frac{2}{c_0}\int_{\{m_\sigma\le \frac{c_0}{2}\}}(m - m_\sigma)(\log m -\log m_\sigma )\,\dx \\&\qquad+
     \int_{\{m_\sigma\ge \frac{c_0}{2}\}} \frac{| m_\sigma -  m|}{\min\{ m_\sigma, m\}}\,\dx \\
     &\leq \frac{2C}{c_0}\sigma + \frac{2}{c_0} \int_{\Tt^d} | m_\sigma -  m|\,\dx.
\end{align*}
This estimate, together with \eqref{eq:convergencem}, yields
\begin{align*}
    \lim_{\sigma\to0}\bigg| - \int_{\Tt^d}(\log m_\sigma - \log m)\,\dx\bigg| =0.
\end{align*}

Consequently,
\[
\lim_{\sigma\to0}\left|\int_{\Tt^d}(u-u_\sigma)\,\dx\right|=0.
\]
Finally, applying the Poincaré--Wirtinger inequality and recalling \eqref{eq:estDuepsi}, we conclude that 
\[
    \|u - u_\sigma\|_{L^2(\Tt^d)} \leq C \left( \|Du - Du_\sigma\|_{L^2(\Tt^d)} + \left| \int_{\Tt^d} (u - u_\sigma) \,\dx \right| \right) \to 0
\]
as $\sigma\to0$, which proves  \eqref{eq:convergenceu}.
\end{proof}

\bibliographystyle{abbrv}
\bibliography{mfgv8_nn, supbib}  

\end{document}